\documentclass[11pt,a4paper,reqno]{amsart}
\usepackage[british]{babel}

\usepackage{a4wide}
\usepackage{setspace}
\usepackage{amssymb}
\usepackage{amsmath}
\usepackage{amsthm}
\usepackage{enumitem}
\usepackage{mathrsfs}
\usepackage[colorlinks,citecolor=blue,urlcolor=black,linkcolor=red,linktocpage]{hyperref}
\usepackage{mathrsfs}
\usepackage{mathtools}

\newtheorem{thm}{Theorem}[section]
\newtheorem{cor}[thm]{Corollary}
\newtheorem{lem}[thm]{Lemma}

\newtheorem{prop}[thm]{Proposition}
\newtheorem{question}[thm]{Question}

\theoremstyle{definition}

\newtheorem{exmpl}[thm]{Example}

\newtheorem{definition}[thm]{Definition}

\newtheorem{remark}[thm]{Remark}

\renewcommand{\epsilon}{\varepsilon}
\renewcommand{\phi}{\varphi}
\newcommand{\defeq}{\mathrel{\mathop:}=}

\DeclareMathOperator{\match}{match}
\DeclareMathOperator{\id}{id}
\DeclareMathOperator{\spt}{spt}

\DeclareMathOperator{\Pfin}{\mathscr{P}_{fin}}

\begin{document}
%%%%%%%%%%%%%%%%%%%%%%%%%%%%%%%%

\setlist{noitemsep}

\author{Kate Juschenko}
\address{K.J., Department of Mathematics, University of Texas at Austin, 2515 Speedway C1200, Austin, TX 78712, USA}
\email{kate.juschenko@gmail.com}

\author{Friedrich Martin Schneider}
\address{F.M.S., Institute of Discrete Mathematics and Algebra, TU Bergakademie Freiberg, 09596 Freiberg, Germany}
\email{martin.schneider@math.tu-freiberg.de}

\title{Skew-amenability of topological groups}
\date{\today}

\begin{abstract} 
  	We study skew-amenable topological groups, i.e., those admitting a left-invariant mean on the space of bounded real-valued functions left-uniformly continuous in the sense of Bourbaki. We prove characterizations of skew-amenability for topological groups of isometries and automorphisms, clarify the connection with extensive amenability of group actions, establish a F\o lner-type characterization, and discuss closure properties of the class of skew-amenable topological groups. Moreover, we isolate a dynamical sufficient condition for skew-amenability and provide several concrete variations of this criterion in the context of transformation groups. These results are then used to decide skew-amenability for a number of examples of topological groups built from or related to Thompson's group $F$ and Monod's group of piecewise projective homeomorphisms of the real line.
\end{abstract}

%\subjclass[2010]{...} 

\maketitle

%%%%%%%%%%%%%%%%%%%%%%%%%%%%%%%%%%%%%%%%%%
%%%%%%%%%%%%%%%%%%%%%%%%%%%%%%%%%%%%%%%%%%

\tableofcontents

\section{Introduction}

A topological group $G$ is called \emph{amenable} if every continuous action of $G$ on a non-void compact Hausdorff space admits an invariant regular Borel probability measure, or equivalently, if every continuous action of $G$ by affine homeomorphisms on a non-void compact convex subset of a locally convex topological vector space has a fixed point. By a fundamental result of Rickert~\cite[Theorem~4.2]{rickert}, a topological group is amenable if and only if there exists a left-invariant mean on the space of right-uniformly continuous bounded real-valued functions on $G$. In his recent analysis of finite-energy path and loop groups~\cite{Pestov2020}, Pestov suggested to study a sibling of amenability---\emph{skew-amenability}---defined along the lines of Rickert's theorem: a topological group $G$ is said to be \emph{skew-amenable} if the space of left-uniformly continuous bounded real-valued functions on $G$ admits a left-invariant mean, which is equivalent to the existence of a right-invariant mean on the space of right-uniformly continuous bounded real-valued functions on $G$. The relevance of this property is rooted in representation theory: by the work of Giordano and Pestov~\cite[Proposition~4.5]{GiordanoPestov}, every strongly continuous unitary representation of a skew-amenable topological group is amenable in the sense of Bekka~\cite{bekka}, and another argument by Pestov~\cite[Proof of Theorem~5.2]{Pestov18} shows that any skew-amenable topological group both having strong property $(T)$ and admitting a topologically faithful unitary representation must be precompact. On the other hand, classical work of Greenleaf~\cite[Theorem~2.2.1]{greenleaf} entails that, within the class of locally compact groups, skew-amenability is equivalent to amenability. Thus, in essence, skew-amenability is a property linked with \emph{infinite-dimensional groups}, i.e., large topological groups occurring as transformation groups in Ramsey theory, operator algebra, and mathematical physics.

The present note is devoted to a comprehensive study of skew-amenability, comprising characterizations of skew-amenability for topological isometry groups, its connection with extensive amenability of group actions, a F\o lner-type characterization, persistence under group-theoretic constructions, and examples. For the latter, particular focus will be on Thompson's group $F$ and Monod's group $H(\mathbb{R})$ of projective homeomorphisms of the real line. As a concrete application of our results, we will answer the following question raised by Pestov.

\begin{question}[Pestov~\cite{Pestov2020}]\label{question:pestov} Is the class of skew-amenable topological groups closed under extensions? \end{question}

Pestov's question has a connection with the notoriously open problem concerning amenability of Thompson's group $F$, which we view as a subgroup of $\mathrm{Aut}(D,{\leq})$ for $D \defeq [0,1] \cap \mathbb{Z}\!\left[ \tfrac{1}{2} \right]$. We prove that $F$ is skew-amenable with respect to the topology of pointwise convergence arising from the discrete topology on $D$ (Corollary~\ref{corollary:F}), whereas its closure $\overline{F} = \mathrm{Aut}(D,{\leq})$ is not (Proposition~\ref{proposition:aut.is.not.skew.amenable}). Moreover, we show that the topological group~$\mathbb{Z}^{(D)} \! \rtimes F$, with $\mathbb{Z}^{(D)}$ carrying the discrete topology and $F$ carrying the topology of pointwise convergence, is skew-amenable if and only if the discrete group $F$ is amenable (Corollary~\ref{corollary:kate}). Consequently, an affirmative answer to Question~\ref{question:pestov} would have implied amenability of the discrete group $F$.

We will answer Question~\ref{question:pestov} in the negative. In fact, we are going to exhibit an example of an action of a group $G$ on a set $X$ such that \begin{itemize}
	\item[---\,] $G$ is skew-amenable with respect to the topology of pointwise convergence induced by its action on the discrete topological space $X$,
	\item[---\,] for every non-trivial amenable group $H$, the associated semi-direct product $H^{(X)} \! \rtimes G$, with $H^{(X)}$ carrying the discrete topology, is not skew-amenable.
\end{itemize} As it turns out, such an example is given by Monod's group $H(\mathbb{R})$ acting on the projective line $\mathbf{P}^{1} \defeq \mathbf{P}^{1}(\mathbb{R})$. In particular, while $H(\mathbb{R})$ is skew-amenable with respect to the topology of pointwise convergence arising from the discrete topology on $\mathbf{P}^{1}$ (Corollary~\ref{corollary:monod}), the topological group $(\mathbb{Z}/2\mathbb{Z})^{(\mathbf{P}^{1})} \! \rtimes H(\mathbb{R})$ is not skew-amenable (Corollary~\ref{corollary:semidirect.monod}). The latter is a consequence of our Theorem~\ref{theorem:extensive} (more concretely, Corollary~\ref{corollary:extensive}) clarifying the relationship between extensive amenability of group actions and skew-amenability of some corresponding topological semi-direct products, and the work of Juschenko, Matte Bon, Monod, and de la Salle~\cite[Theorem~6.1]{JMMS} establishing non-extensive amenability of the considered action.

The concrete results described above are instances of some general dynamical phenomena. Motivated by the considered examples, we isolate the notion of \emph{proximal simulation} (Definition~\ref{definition:proximal.simulation}(1)). Roughly speaking, a topological group $G$ is proximally simulated by a subgroup $H \leq G$ if the action of $G$ on itself by left-translations can be approximately mimicked, upon composition with a suitable right-translation, by the left-translation action of $H$ on~$G$. Using a weak form of this concept, \emph{weak proximal simulation} (Definition~\ref{definition:proximal.simulation}(2)), we establish an abstract sufficient criterion for skew-amenability (Proposition~\ref{proposition:weak.proximal.simulation.implies.amenability}), which entails the following: any topological group weakly proximally simulated by some discretely amenable subgroup is skew-amenable (Corollary~\ref{corollary:weak.proximal.simulation.implies.amenability}). This result admits a number of rather concrete variations for topological groups of isometries and permutations (Proposition~\ref{proposition:approximation.implies.skew.amenability} and Corollary~\ref{corollary:approximation.implies.skew.amenability}--\ref{corollary:permutation.groups}) and, in particular, applies to Monod's group $H(\mathbb{R})$ and Thompson's group $F$.

This article is organized as follows. In Section~\ref{section:preliminaries}, we set up some preliminaries and clarify some notation concerning invariant means and uniform spaces. The following Section~\ref{section:skew.amenability} contains a characterization of skew-amenability for isometry groups and permutation groups, which will be applied in Section~\ref{section:extensive} to clarify the connection between extensive amenability of group actions and skew-amenability of certain topological groups naturally associated to these. In Section~\ref{section:folner}, we note a characterization of skew-amenability of topological groups in the spirit of F\o lner's amenability criterion for discrete groups. The subsequent Section~\ref{section:persistence} is devoted to a discussion of permanence properties of the class of skew-amenable topological groups. In Section~\ref{section:proximal.simulation}, we provide a method for proving skew-amenability using our concept of (weak) proximal simulation. In the final Sections~\ref{section:monod} and~\ref{section:thompson}, we decide skew-amenability for the above-mentioned examples.

\section{Preliminaries: means and uniform spaces}\label{section:preliminaries}

\subsection{General notation} We start off with some general notation and terminology. Given a normed space $E$, we will denote by $E^{\ast}$  the \emph{continuous dual} of $E$, i.e., the Banach space of all continuous linear maps from $E$ to its underlying scalar field. Now, let $X$ be a set. We will denote by $\mathscr{P}(X)$ the power set of $X$ and by $\Pfin (X)$ the set of all finite subsets of $X$. Furthermore, we will consider the unital real Banach algebra $\ell^{\infty}(X)$ of all bounded real-valued functions on $X$ equipped with the supremum norm \begin{displaymath}
	\Vert f \Vert_{\infty} \, \defeq \, \sup \{ \vert f(x) \vert \mid x \in X \} \qquad (f \in \ell^{\infty}(X)) 
\end{displaymath} as well as the real Banach space $\ell^{1}(X)$ of all unconditionally summable real-valued functions on $X$ endowed with the norm \begin{displaymath}
	\Vert f \Vert_{\infty} \, \defeq \, \sum\nolimits_{x \in X} \vert f(x) \vert \qquad (f \in \ell^{1}(X)) \, .
\end{displaymath} For any $f \in \mathbb{R}^{X}$, we denote $\spt (f) \defeq f^{-1}(\mathbb{R}\setminus\{ 0 \})$. Moreover, we let \begin{displaymath}
	\Delta (X) \, \defeq \, \!\left\{ \mu \in [0,1]^{X} \left\vert \, \spt (\mu) \text{ finite}, \, \sum\nolimits_{x \in X} \mu (x) = 1 \right\} \! \, \subseteq \, \ell^{1}(X) \, . \right.     
\end{displaymath} As usual, for each $x \in X$, we define $\delta_{x} \in \Delta (X)$ by \begin{displaymath}
	\delta_{x}(y) \, \defeq \, \begin{cases}
			\, 1 & \text{for } y=x \, , \\
			\, 0 & \text{for } y \in X \setminus \{ x \} \, .
		\end{cases}
\end{displaymath} If $F$ is a non-empty finite subset of $X$, then we define \begin{displaymath}
	\delta_{F} \, \defeq \, \tfrac{1}{\vert F \vert}\sum\nolimits_{x \in F} \delta_{x} \, \in \, \Delta (X) .
\end{displaymath} Now, consider any linear subspace $A \subseteq \ell^{\infty}(X)$ containing $\mathbf{1} \colon A \to \mathbb{R}, \, x \mapsto 1$. Then a \emph{mean} on $A$ is a positive linear functional $\mu \colon A \to \mathbb{R}$ such that $\mu (\mathbf{1}) = 1$. The set of all means on $A$ will be denoted by $\mathrm{M}(A)$. If $A$ separates the elements of $X$, in the sense that $\sup_{f \in A} \vert f(x) - f(y) \vert > 0$ for any two distinct $x,y \in X$, then we may and will consider $\Delta (X)$ as a subset of $\mathrm{M}(A)$ by identifying each element $\mu \in \Delta (X)$ with the corresponding mean \begin{displaymath}
	A \, \longrightarrow \, \mathbb{R}, \quad f \, \longmapsto \, a(f) \, \defeq \, \sum\nolimits_{x \in X} f(x)\mu(x) \, .
\end{displaymath}

\subsection{Invariant means} We clarify some notation and terminology regarding group actions. For this purpose, let $\alpha \colon G \times X \to X$ be an action of a group $G$ on a set $X$. For $g \in G$, let $\alpha_{g} \colon X \to X, \, x \mapsto \alpha (g,x)$. Of course, as is customary, if the action $\alpha$ is understood, then we will simply write $gx \defeq \alpha_{g}(x)$ whenever $g \in G$ and $x \in X$. A subset $A \subseteq \ell^{\infty}(X)$ will be called \emph{$\alpha$-invariant} or \emph{$G$-invariant}\footnote{provided that there is no possible ambiguity concerning the action} if $f \circ \alpha_{g} \in A$ for all $f \in A$ and $g \in G$. Considering an $\alpha$-invariant linear subspace $A \subseteq \ell^{\infty}(X)$, a functional $\mu \in A^{\ast}$ will be called \emph{$\alpha$-invariant} or \emph{$G$-invariant}\footnotemark[1] if $\mu (f \circ \alpha_{g}) = \mu (f)$ for all $f \in A$ and $g \in G$. Moreover, for $g \in G$, let $\lambda_{g} \colon G \to G, \, x \mapsto gx$ and $\rho_{g} \colon G \to G, \, x \mapsto xg$. A subset $A \subseteq \ell^{\infty}(G)$ is called \emph{left-invariant} (resp.,~\emph{right-invariant}) if $f \circ \lambda_{g} \in A$ (resp.,~$f \circ \rho_{g} \in A$) for all $f \in A$ and $g \in G$. A subset of $\ell^{\infty}(G)$ is called \emph{bi-invariant} if it is both left- and right-invariant. Given a left-invariant (resp.,~right-invariant) linear subspace $A \subseteq \ell^{\infty}(X)$, we call an element $\mu \in A^{\ast}$ \emph{left-invariant} (resp.,~\emph{right-invariant}) if $\mu (f \circ \lambda_{g}) = \mu (f)$ (resp.,~$\mu (f \circ \rho_{g}) = \mu (f)$) for all $f \in A$ and $g \in G$.

\begin{lem}\label{lemma:directed.union.amenable} Let $\alpha \colon G \times X \to X$ be an action of a group $G$ on a set $X$ and let $\mathscr{A}$ be an upward directed set of $\alpha$-invariant linear subspaces of $\ell^{\infty}(X)$ with $\mathbf{1} \in \bigcap \mathscr{A}$. If every member of $\mathscr{A}$ admits an $\alpha$-invariant mean, then so does the norm-closure of $\bigcup \mathscr{A}$ in $\ell^{\infty}(X)$. \end{lem}

\begin{proof} Let $C$ denote the norm-closure of $\bigcup \mathscr{A}$ in $\ell^{\infty}(X)$. As $\mathrm{M}(C)$ is compact with regard to the weak-$\ast$ topology (see, e.g., \cite[2.1, Theorem~1.8(i), p.~68]{AnalysisOnSemigroups}), it suffices to deduce that \begin{align*}
	\forall \epsilon > 0 \ \forall E \in \Pfin (G) \ \forall F \in \Pfin &(C) \ \exists \mu \in \Delta (X) \colon \\
		& \sup\nolimits_{g \in E} \sup\nolimits_{f \in F} \vert \mu(f) - \mu (f \circ \alpha_{g}) \vert \, \leq \, \epsilon \, .
\end{align*} For this purpose, let $\epsilon > 0$, $E \in \Pfin (G)$ and $F \in \Pfin (C)$. By definition of $C$ and upward directedness of $\mathscr{A}$, we find $A \in \mathscr{A}$ and a family $(f^{\ast})_{f \in F} \in A^{F}$ such that $\Vert f^{\ast} -f \Vert_{\infty} \leq \tfrac{\epsilon}{3}$ for each $f \in F$. By our hypothesis, $A$ admits an $\alpha$-invariant mean. Since $\Delta (X)$ is weak-$\ast$ dense in $\mathrm{M}(A)$ (see, e.g., \cite[2.1, Theorem~1.8(ii), p.~68]{AnalysisOnSemigroups}), there exists some $\mu \in \Delta (X)$ such that $\vert \mu (f^{\ast}) - \mu (f^{\ast} \circ \alpha_{g}) \vert \leq \tfrac{\epsilon}{3}$ for all $f \in F$ and $g \in E$. We conclude that \begin{align*}
	\vert \mu (f) - \mu (f \circ \alpha_{g}) \vert \, &\leq \, \vert \mu(f) - \mu (f^{\ast}) \vert + \vert \mu (f^{\ast}) - \mu (f^{\ast} \circ \alpha_{g}) \vert + \vert \mu(f^{\ast} \circ \alpha_{g}) - \mu(f^{\ast} \circ \alpha_{g}) \vert \\
		& \leq \, \Vert f^{\ast} -f \Vert_{\infty} + \vert \mu (f^{\ast}) - \mu (f^{\ast} \circ \alpha_{g}) \vert + \Vert (f^{\ast} \circ \alpha_{g}) - (f \circ \alpha_{g}) \Vert_{\infty} \, \leq \, \epsilon 
\end{align*} whenever $f \in F$ and $g \in E$, as desired. \end{proof}

We are going to apply Lemma~\ref{lemma:directed.union.amenable} in the proof of Proposition~\ref{proposition:isometry.group}. Furthermore, let us note that Lemma~\ref{lemma:directed.union.amenable} can be used to prove the following result of Pestov~\cite{Pestov12}: the inverse limit of any directed family of amenable topological groups, whose limit projections have dense images, is again amenable.

\subsection{Uniform spaces} For the sake of convenience, we recollect some standard terminology concerning uniform spaces. A \emph{uniformity} on a set $X$ is a filter $\mathscr{E}$ on the set $X \times X$ such that \begin{enumerate}
	\item[---\,] $\{ (x,x) \mid x \in X \} \subseteq E$ for every $E \in \mathscr{E}$,
	\item[---\,] $E^{-1} \in \mathscr{E}$ for every $E \in \mathscr{E}$,
	\item[---\,] for every $E_{0} \in \mathscr{E}$ there exists $E_{1} \in \mathscr{E}$ such that $E_{1} \circ E_{1} \subseteq E_{0}$.
\end{enumerate} A \emph{uniform space} is a set $X$ equipped with a uniformity, the elements of which are called the \emph{entourages} of the uniform space $X$. For a uniform space $X$, the \emph{induced topology} on $X$ is defined as follows: a subset $S \subseteq X$ is \emph{open} in $X$ if, for every $x \in S$, there is an entourage $E$ of $X$ such that $\{ y \in X \mid (x,y) \in E \} \subseteq S$. It is well known that every compact Hausdorff space admits a unique uniformity inducing its topology~\cite[II, \S 4.1, Theorem~1]{bourbaki}. Furthermore, any metric space naturally gives rise to a uniform space: if $(X,d)$ is a metric space, then \begin{displaymath}
	\mathscr{E}_{d} \, \defeq \, \{ E \subseteq X \times X \mid \exists \epsilon > 0 \, \forall x,y \in X \colon \, d(x,y) < \epsilon \Longrightarrow (x,y) \in E \}
\end{displaymath} constitutes a uniformity on the set $X$ inducing the topology generated by $d$, and we will occasionally consider $(X,d)$ as a uniform space by referring to $(X,\mathscr{E}_{d})$.

Let $X$ and $Y$ be uniform spaces. A map $f \colon X \to Y$ is called \emph{uniformly continuous} if for every entourage $E$ of~$Y$ there is an entourage $F$ of~$X$ such that $\{ (f(x),f(y)) \mid (x,y) \in F \} \subseteq E$. A bijection $f \colon X \to Y$ is called an \emph{isomorphism} if both $f$ and $f^{-1}$ are uniformly continuous maps. A set $H \subseteq Y^{X}$ is said to be \emph{uniformly equicontinuous} if for every entourage $F$ of $Y$ there exists an entourage $E$ of~$X$ such that $\{ (f(x),f(y)) \mid (x,y) \in E \} \subseteq F$ whenever $f \in H$. Let us point out that a set $H \subseteq \mathbb{R}^{X}$ is uniformly equicontinuous continuous with respect to the uniformity on $\mathbb{R}$ induced by the Euclidean metric if and only if for every $\epsilon > 0$ there exists an entourage $E$ of $X$ such that \begin{displaymath}
	\forall f \in H \ \forall (x,y) \in E \colon \qquad \vert f(x) - f(y) \vert \, \leq \, \epsilon .
\end{displaymath} In particular, a function $f \colon X \to \mathbb{R}$ is uniformly continuous if and only if for every $\epsilon > 0$ there exists an entourage $E$ of $X$ such that \begin{displaymath}
	\forall (x,y) \in E \colon \qquad \vert f(x) - f(y) \vert \, \leq \, \epsilon .
\end{displaymath} The set $\mathrm{UCB}(X)$ of all uniformly continuous, bounded real-valued functions on~$X$ is a closed unital subalgebra of $\ell^{\infty}(X)$. A subset $H \subseteq \mathrm{UCB}(X)$ is called \emph{UEB} (shorthand for \emph{uniformly equicontinuous, bounded}) if $H$ is uniformly equicontinuous and $\Vert \cdot \Vert_{\infty}$-bounded. The collection $\mathrm{UEB}(X)$ of all UEB subsets of $\mathrm{UCB}(X)$ constitutes a convex vector bornology on $\mathrm{UCB}(X)$. The set $\mathrm{UEB}(X)$ is an important tool in Pachl's monograph~\cite{PachlBook}.

\section{Skew-amenability}\label{section:skew.amenability}

In this section we will define skew-amenability and provide a simple characterization of this property in the context of topological groups of isometries (Proposition~\ref{proposition:isometry.group}). For convenience in later considerations, we will start by compiling some necessary background on topological groups and uniform spaces.

Let $G$ be a topological group. We will denote by $\mathscr{U}(G)$ the set of all neighborhoods of the neutral element in $G$. For $U \in \mathscr{U}(G)$, let us define \begin{displaymath}
	U_{\Rsh} \, \defeq \, \left\{ (x,y) \in G \times G \left\vert \, xy^{-1} \in U \right\} \! , \right. \qquad U_{\Lsh} \, \defeq \, \left\{ (x,y) \in G \times G \left\vert \, x^{-1}y \in U \right\} \! . \right.
\end{displaymath} The following two uniformities are naturally associated with $G$: both the \emph{right uniformity} \begin{displaymath}
	\mathscr{E}_{\Rsh}(G) \, \defeq \, \left\{ E \subseteq G \times G \left\vert \, \exists U \in \mathscr{U}(G) \colon \, U_{\Rsh} \subseteq E \right\} \right.
\end{displaymath} and the \emph{left uniformity} \begin{displaymath}
	\mathscr{E}_{\Lsh}(G) \, \defeq \, \left\{ E \subseteq G \times G \left\vert \, \exists U \in \mathscr{U}(G) \colon \, U_{\Lsh} \subseteq E \right\}  \right.
\end{displaymath} induce the original topology of~$G$. As would seem natural, a real-valued function $f \colon G \to \mathbb{R}$ will be called \emph{right-uniformly continuous} (resp.,~\emph{left-uniformly continuous}) if $f$ is uniformly continuous with respect to $\mathscr{E}_{\Rsh}(G)$ (resp.,~$\mathscr{E}_{\Lsh}(G)$), and a subset $H \subseteq \mathbb{R}^{G}$ will be called \emph{right-uniformly equicontinuous} (resp.,~\emph{left-uniformly equicontinuous}) if $H$ is uniformly equicontinuous with respect to $\mathscr{E}_{\Rsh}(G)$ (resp.,~$\mathscr{E}_{\Lsh}(G)$). Of course, both of the associated closed unital subalgebras of~$\ell^{\infty}(G)$, that is, \begin{displaymath}
	\mathrm{RUCB}(G) \, \defeq \, \mathrm{UCB}(G,\mathscr{E}_{\Rsh}(G)) , \qquad \mathrm{LUCB}(G) \, \defeq \, \mathrm{UCB}(G,\mathscr{E}_{\Lsh}(G))
\end{displaymath} are bi-invariant. Furthermore, we abbreviate \begin{displaymath}
	\mathrm{RUEB}(G) \, \defeq \, \mathrm{UEB}(G,\mathscr{E}_{\Rsh}(G)) , \qquad \mathrm{LUEB}(G) \, \defeq \, \mathrm{UEB}(G,\mathscr{E}_{\Lsh}(G)) .
\end{displaymath} As usual, we will say that $G$ has \emph{small invariant neighborhoods} or call $G$ a \emph{SIN} group if \begin{displaymath}
	\forall U \in \mathscr{U}(G) \colon \quad \bigcap\nolimits_{g \in G} g^{-1}Ug \, \in \, \mathscr{U}(G) .
\end{displaymath} It is easy to see that $G$ has small invariant neighborhoods if and only if $\mathscr{E}_{\Rsh}(G) = \mathscr{E}_{\Lsh}(G)$. In particular, if $G$ has small invariant neighborhoods, then $\mathrm{RUCB}(G) = \mathrm{LUCB}(G)$.

We note the following compatibility properties of the two uniformities introduced above with respect to subgroups and quotients.

\begin{remark}\label{remark:induced.uniformities} Let $G$ be a topological group. \begin{itemize}
	\item[$(1)$] Let $H$ be a topological subgroup of $G$. The right (resp., left) uniformity of $H$ coincides with the subspace uniformity inherited from $(G,\mathscr{E}_{\Rsh}(H))$ (resp., $(G,\mathscr{E}_{\Lsh}(H))$), i.e., \begin{align*}
					\qquad \quad  \mathscr{E}_{\Rsh}(H) \, &= \, \{ E \cap (H \times H) \mid E \in \mathscr{E}_{\Rsh}(G) \} , \\
					\mathscr{E}_{\Lsh}(H) \, &= \, \{ E \cap (H \times H) \mid E \in \mathscr{E}_{\Lsh}(G) \} .
				\end{align*} Thanks to the well-known extension theorem for uniformly continuous bounded real-valued functions (see, e.g.,~\cite[Theorem~2.22]{PachlBook}), this moreover entails that \begin{align*}
					\qquad \quad \mathrm{RUCB}(H) \, &= \, \{ {f\vert_{H}} \mid f \in \mathrm{RUCB}(G) \} , \\
					\mathrm{LUCB}(H) \, &= \, \{ {f\vert_{H}} \mid f \in \mathrm{LUCB}(G) \} .
				\end{align*}
	\item[$(2)$] Let $N$ be a normal subgroup of $G$. Then $G/N$ constitutes a topological group with respect to the quotient topology, that is, the final topology induced by the projection $\pi_{N} \colon G \to G/N, \, x \mapsto xN$. It is straightforward to check that \begin{align*}
					\qquad \quad  \mathscr{E}_{\Rsh}(G/N) \, &= \, \{\{ (xN,yN) \mid (x,y) \in U_{\Rsh} \} \mid U \in \mathscr{U}(G) \} , \\
					\mathscr{E}_{\Lsh}(G/N) \, &= \, \{\{ (xN,yN) \mid (x,y) \in U_{\Lsh} \} \mid U \in \mathscr{U}(G) \} ,
				\end{align*} and that the maps \begin{align*}
					\qquad \quad &\mathrm{RUCB}(G/N) \, \longrightarrow \, \{ h \in \mathrm{RUCB}(G) \mid \forall g \in G \colon \, h \circ {\rho_{g}} = h \} , \quad f \, \longmapsto \, f \circ {\pi_{N}} , \\
					 &\mathrm{LUCB}(G/N) \, \longrightarrow \, \{ h \in \mathrm{LUCB}(G) \mid \forall g \in G \colon \, h \circ {\rho_{g}} = h \} , \quad f \, \longmapsto \, f \circ {\pi_{N}}
				\end{align*} are isometric isomorphisms of unital real Banach algebras.
\end{itemize} \end{remark}

For later reference, let us briefly recollect some standard terminology concerning actions of topological groups on uniform spaces. An action $\alpha \colon G \times X \to X$ of a topological group $G$ on a uniform space $X$ will be called \emph{uniformly equicontinuous} if \begin{enumerate}
	\item[---\,] the set $\{ \alpha_{g} \mid g \in G \}$ is uniformly equicontinuous, and
	\item[---\,] for each $x \in X$, the map $G \to X, \, g \mapsto \alpha (g,x)$ is continuous.
\end{enumerate} Evidently, the second condition is void for actions of discrete groups, whereas the first condition is void for actions on discrete uniform spaces. It is not difficult to see that an action $\alpha \colon G \times X \to X$ of a topological group $G$ on a uniform space~$X$ is uniformly equicontinuous if and only if \begin{enumerate}
	\item[---\,] the set $\{ \alpha_{g} \mid g \in G \}$ is uniformly equicontinuous, and
	\item[---\,] the map $\alpha \colon G \times X \to X$ is continuous.
\end{enumerate} Also, let us note that a continuous action of a topological group $G$ by automorphisms on a topological group $H$ is uniformly equicontinuous with respect to the right (resp., left) uniformity on $H$ if and only if \begin{displaymath}
	\forall U \in \mathcal{U}(H) \ \exists V \in \mathcal{U}(H) \ \forall g \in G \colon \quad \alpha_{g}(V) \subseteq U .
\end{displaymath} Furthermore, we recall that an action $\alpha \colon G \times X \to X$ of a topological group $G$ by isomorphisms on a uniform space $X$ is said to be \emph{bounded}~\cite[Definition~3.2]{vries75} if, for every entourage $E$ of $X$, there exists $U \in \mathscr{U}(G)$ such that \begin{displaymath}
	\forall g \in U \ \forall x \in X \colon \quad (x,\alpha_{g}(x)) \in E .
\end{displaymath} It is easily verified that an action $\alpha \colon G \times X \to X$ of a topological group $G$ by isomorphisms on a uniform space $X$ is bounded if and only if the associated homomorphism $G \to \mathrm{Aut}(X), \, g \mapsto \alpha_{g}$ is continuous with respect to the topology of uniform convergence on $\mathrm{Aut}(X)$. In particular, every bounded action of a topological group by isomorphisms on a uniform space is necessarily continuous. Moreover, let us note that an action $\alpha$ of a topological group $G$ by automorphisms on a topological group $H$ is bounded with respect to the right uniformity on $H$ if and only if $\alpha$ is bounded with respect to the left uniformity on $H$.

\begin{remark}[cf.~\cite{vries75}, Proposition~3.3]\label{remark:vries} Every continuous action of a topological group on a compact Hausdorff space $X$ is bounded with respect to the unique uniformity inducing the topology of~$X$. \end{remark}

For later reference, we briefly recollect some well-known facts about precompact groups. Recall that a topological group $G$ is \emph{precompact} if, for every $U \in \mathscr{U}(G)$, there exists a finite subset $S \subseteq G$ such that $G = US$.

\begin{lem}\label{lemma:precompact} Every bounded action of a precompact topological group by isomorphisms on a uniform space is uniformly equicontinuous. \end{lem}

\begin{proof} Consider any bounded action of a precompact topological group $G$ by isomorphisms on a uniform space $X$. To prove uniform equicontinuity, let $E_{0}$ be an entourage of $X$. Then we find an entourage $E_{1}$ of $X$ such that $E_{1}^{-1} = E_{1}$ and ${E_{1}} \circ {E_{1}} \circ {E_{1}} \subseteq {E_{0}}$. Since the action is bounded, there exists $U \in \mathscr{U}(G)$ such that \begin{displaymath}
	\forall g \in U \ \forall x \in X \colon \quad (x,gx) \in E_{1} .
\end{displaymath} Being precompact, $G$ admits a finite subset $S \subseteq G$ such that $G=US$. As $G$ acts on $X$ by isomorphisms, there exists an entourage $E_{2}$ of $X$ such that \begin{displaymath}
	\forall s \in S \ \forall (x,y) \in E_{2} \colon \quad (sx,sy) \in E_{1} .
\end{displaymath} We now show that \begin{equation}\tag{$\ast$}\label{precompact}
	\forall g \in G \ \forall (x,y) \in E_{2} \colon \quad (gx,gy) \in E_{3} .
\end{equation} To this end, let $g \in G$. Then there exists $s \in S$ with $g \in Us$. We conclude that $(sx,gx) \in E_{1}$ and thus also $(gx,sx) \in E_{1}^{-1} = E_{1}$ for every $x \in X$. Hence, if $(x,y) \in E_{2}$, then \begin{displaymath}
	(gx,sx) \in E_{1} , \quad (sx,sy) \in E_{1} , \quad (sy,gy) \in E_{1} ,
\end{displaymath} thus $(gx,gy) \in {E_{1}} \circ {E_{1}} \circ {E_{1}} \subseteq {E_{0}}$. This proves~\eqref{precompact} and therefore completes the argument. \end{proof}

\begin{cor}\label{corollary:precompact} Every continuous action of a precompact topological group on a compact Hausdorff space~$X$ is uniformly equicontinuous with respect to the unique uniformity inducing the topology of~$X$. \end{cor}

\begin{proof} This is an immediate consequence of Remark~\ref{remark:vries} and Lemma~\ref{lemma:precompact}. \end{proof}

\begin{remark}\label{remark:precompact.amenable} Let $G$ be a topological group. The convex subset $\mathrm{M}(\mathrm{RUCB}(G)) \subseteq \mathrm{RUCB}(G)^{\ast}$ is compact with respect to the weak-$\ast$ topology (see, e.g., \cite[2.1, Theorem~1.8(i), p.~68]{AnalysisOnSemigroups}). Furthermore, the affine action of $G$ on $\mathrm{M}(\mathrm{RUCB}(G))$ given by \begin{displaymath}
	(g\mu)(f) \, \defeq \, \mu (f \circ \lambda_{g}) \qquad (g \in G, \, \mu \in \mathrm{M}(\mathrm{RUCB}(G)), \, f \in \mathrm{RUCB}(G))
\end{displaymath} is continuous with respect to the weak-$\ast$ topology~\cite[Corollary~9.36(3)]{PachlBook}. Therefore, if $G$ is precompact, then this action is uniformly equicontinuous (with regard to the unique uniformity on $\mathrm{M}(\mathrm{RUCB}(G))$ inducing the weak-$\ast$ topology) by Corollary~\ref{corollary:precompact}, hence admits a fixed point by Kakutani's fixed point theorem~\cite{Kakutani38} (see~\cite[I, Theorem~5.11]{rudin} for a complete proof). This shows that every precompact topological group is necessarily amenable. \end{remark}

Before getting to skew-amenability, let us also recall the concept of amenability for topological groups. A topological group $G$ is called \emph{amenable} if every continuous action of $G$ on a non-void compact Hausdorff space admits an invariant regular Borel probability measure, which is equivalent to saying that every continuous action of $G$ by affine homeomorphisms on a non-void compact convex subset of a locally convex topological vector space has a fixed point. Due to a well-known result of Rickert~\cite[Theorem~4.2]{rickert}, a topological group $G$ is amenable if and only if $\mathrm{RUCB}(G)$ admits a left-invariant mean. For later reference, we furthermore recall that a topological group $G$ is said to be \emph{extremely amenable} if every continuous action of $G$ on a non-empty compact Hausdorff space has a fixed point. In analogy with Rickert's work, a topological group $G$ is extremely amenable if and only if $\mathrm{RUCB}(G)$ admits a left-invariant multiplicative mean. For a comprehensive account on (extreme) amenability of general topological groups, we refer to~\cite{PestovBook,GrigorchukDeLaHarpe}.

\begin{definition}[Pestov~\cite{Pestov2020}] A topological group $G$ is called \emph{skew-amenable} if $\mathrm{LUCB}(G)$ admits a left-invariant mean. \end{definition} 

Of course, for topological groups having small invariant neighborhoods (such as discrete groups), skew-amenability is clearly equivalent to amenability. And due to classical work of Greenleaf~\cite[Theorem~2.2.1]{greenleaf}, a locally compact group is skew-amenable if and only if it is amenable. Furthermore, if a topological group $G$ is \emph{B-amenable}~\cite[Definition~3.8]{GrigorchukDeLaHarpe}, i.e., the space of all continuous bounded real-valued functions on $G$ admits a left-invariant mean, then $G$ is both amenable and skew-amenable. For more information on B-amenability of topological groups (including historic background), the reader is referred to~\cite{GrigorchukDeLaHarpe}. We continue with the a basic remark.

\begin{remark}\label{remark:basic} Let $G$ be a topological group and consider the map $\iota \colon G \to G, \, x \mapsto x^{-1}$. Since \begin{displaymath}
	I \colon \, \mathrm{RUCB}(G) \, \longrightarrow \, \mathrm{LUCB}(G) , \qquad f \, \longmapsto \, f \circ \iota
\end{displaymath} is an isometric isomorphism of unital real Banach algebras and \begin{displaymath}
	\forall g \in G \ \forall f \in \mathrm{RUCB}(G) \colon \qquad I(f \circ \lambda_{g}) \, = \, I(f) \circ \rho_{g^{-1}}, \quad I(f \circ \rho_{g}) \, = \, I(f) \circ \lambda_{g^{-1}} ,
\end{displaymath} the following statements hold: \begin{itemize}
	\item[$(1)$] $G$ is amenable if and only if $\mathrm{LUCB}(G)$ admits a right-invariant mean.
	\item[$(2)$] $G$ is skew-amenable if and only if $\mathrm{RUCB}(G)$ admits a right-invariant mean.
\end{itemize} \end{remark} 

The following analogue of Rickert's fixed point theorem~\cite[Theorem~4.2]{rickert} in the context of skew-amenability has been kindly pointed out to the authors by Jan Pachl. Recall that an action of a topological group $G$ by homeomorphisms on a compact Hausdorff space $X$ is said to be \emph{slightly continuous}~\cite{day} if there exists at least one element $x \in X$ such that the map $G \to X, \, g \mapsto gx$ is continuous. Extending this terminology, an action of a topological group $G$ by homeomorphisms on a compact Hausdorff space~$X$ is called \emph{slightly left-uniformly continuous}~\cite{pachl} if there exists an element $x \in X$ such that $G \to X, \, g \mapsto gx$ is uniformly continuous with respect to the left uniformity on $G$ and the unique uniformity on $X$ compatible with its compact Hausdorff topology.

\begin{remark}[cf.~\cite{pachl}, Theorem~3.2]\label{remark:skew.rickert} A topological group $G$ is skew-amenable if and only if every slightly left-uniformly continuous action of $G$ by affine homeomorphisms compact convex subset of a locally convex topological vector space admits a fixed point. This follows from the equivalence of (i) and (iv) in~\cite[Theorem~3.2]{pachl}, evaluated for the left uniformity. \end{remark}

Let us mention the following characterization of skew-amenability by means of uniformly equicontinuous actions, which is implicit in~\cite[Sect.~3.5--3.6]{PestovBook}

\begin{prop}[\cite{PestovBook}, Sections~3.5, 3.6]\label{proposition:skew.amenability} A topological group $G$ is skew-amenable if and only if, for every uniformly equicontinuous action of $G$ on a uniform space $X$, there exists a $G$-invariant mean on~$\mathrm{UCB}(X)$. \end{prop} 

\begin{proof} ($\Longleftarrow$) This is due to the fact that the action $G \times X \to X, \, (g,x) \mapsto \lambda_{g}(x) = gx$ is uniformly equicontinuous with respect to the left uniformity on $X \defeq G$.

($\Longrightarrow$) Let $\alpha$ be a uniformly equicontinuous action of a skew-amenable topological group $G$ on a uniform space $X$. Pick any point $x \in X$. A straightforward argument shows that \begin{displaymath}
	\Phi \colon \, \mathrm{UCB}(X) \, \longrightarrow \, \mathrm{LUCB}(G), \qquad f \, \longmapsto \, \bigl(g \mapsto f(\alpha(g,x))\bigr)
\end{displaymath} constitutes a well-defined, positive, unital linear operator satisfying $\Phi (f \circ \alpha_{g}) = \Phi(f) \circ \lambda_{g}$ for all $f \in \mathrm{UCB}(X)$ and $g \in G$. (We refer to~\cite[Lemma~3.6.5]{PestovBook} for a detailed exposition.) Consequently, if $\mu \colon \mathrm{LUCB}(G) \to \mathbb{R}$ is any left-invariant mean, then $\mu \circ \Phi \colon \mathrm{UCB}(X) \to \mathbb{R}$ will be an $\alpha$-invariant mean. \end{proof}

Since any continuous action of a topological group on a discrete topological space $X$ is necessarily uniformly equicontinuous with respect to the discrete uniformity on $X$, Proposition~\ref{proposition:skew.amenability} readily entails the following.

\begin{cor}\label{corollary:discrete.skew.amenability} Let $G$ be a skew-amenable topological group. If $G$ acts continuously on a discrete topological space~$X$, then there exists a $G$-invariant mean on~$\ell^{\infty}(X)$. \end{cor}
 
Our next purpose is to provide a convenient characterization of skew-amenability for topological groups of isometries. Given a metric space $X$, we consider the topological group~$\mathrm{Iso}(X)$, consisting of all isometries of $X$, endowed with the topology of pointwise convergence with respect to the metric on $X$. We recall that, if $G$ is a topological subgroup of $\mathrm{Iso}(X)$ for some metric space $X$, then the sets of the form \begin{displaymath}
	U_{G}(F,\epsilon) \, \defeq \, \{ g \in G \mid \forall x \in F \colon \, d_{X}(x,gx) < \epsilon \} \qquad (F \in \Pfin (X), \, \epsilon > 0)
\end{displaymath} constitute a neighborhood basis of the identity in $G$. For such groups, skew-amenability may be detected as described in Proposition~\ref{proposition:isometry.group}. As usual, if $G$ is a group acting on a set $X$, then for any $n \in \mathbb{N}$ we may consider the induced action of $G$ on $X^{n}$ given by $gx \defeq (gx_{0},\ldots,gx_{n-1})$ for all $g \in G$ and $x = (x_{0},\ldots,x_{n-1}) \in X^{n}$.

\begin{prop}\label{proposition:isometry.group} Let $X$ be a metric space and let $G$ be a topological subgroup of $\mathrm{Iso}(X)$. Then $G$ is skew-amenable if and only if, for every $n \in \mathbb{N}$ and every $x \in X^{n}$, there exists a $G$-invariant mean on~$\mathrm{UCB}(Gx)$. \end{prop}

To prove Proposition~\ref{proposition:isometry.group}, we recall a well-known description of the algebra of left-uniformly continuous bounded real-valued functions on an arbitrary isometry group. Given a group $G$ acting on a set $X$, we will let ${f\!\!\upharpoonright_{x}} \colon G \to \mathbb{R}, \, g \mapsto f(gx)$ for any $x \in X$ and $f \in \mathbb{R}^{X}$.

\begin{lem}[see, e.g., \cite{SchneiderThomRandomWalks}, Lemma~5.3]\label{lemma:isometry.groups} Let $X$ be a metric space. If $G$ is any topological subgroup of $\mathrm{Iso}(X)$, then the linear subspace \begin{displaymath}
	\left\{ {f\!\!\upharpoonright_{x}} \left\vert \, n \in \mathbb{N}, \, x \in X^{n}, \, f \in \mathrm{UCB}(Gx) \right\} \right.
\end{displaymath} is norm-dense in $\mathrm{LUCB}(G)$. \end{lem}

\begin{proof} This follows from~\cite[Lemma~5.3]{SchneiderThomRandomWalks} and Remark~\ref{remark:basic}. \end{proof}

Combining Lemma~\ref{lemma:isometry.groups} with Lemma~\ref{lemma:directed.union.amenable}, we can now prove Proposition~\ref{proposition:isometry.group}.

\begin{proof}[Proof of Proposition~\ref{proposition:isometry.group}] ($\Longrightarrow$) This is an immediate consequence of Proposition~\ref{proposition:skew.amenability}.
	
($\Longleftarrow$) Note that \begin{displaymath}
	\left. \mathscr{A} \, \defeq \, \left\{ \{ {f\!\!\upharpoonright_{x}} \mid f \in \mathrm{UCB}(Gx) \} \, \right\vert n \in \mathbb{N}, \, x \in X^{n} \right\} 
\end{displaymath} constitutes an upward directed set of linear subspaces of $\mathrm{LUCB}(G)$ such that $\mathbf{1} \in \bigcap \mathscr{A}$. Due to Lemma~\ref{lemma:isometry.groups}, thus $\bigcup \mathscr{A}$ is norm-dense in $\mathrm{LUCB}(G)$. By hypothesis, every member of $\mathscr{A}$ admits a left-invariant mean. Thus, $G$ is skew-amenable by Lemma~\ref{lemma:directed.union.amenable}. \end{proof}

Specializing Proposition~\ref{proposition:isometry.group} to the discrete setting, we obtain the subsequent corollaries. Given a set $X$, we consider the topological group $\mathrm{Sym}(X)$, the full symmetric group over $X$, endowed with the topology of pointwise convergence associated with the discrete topology on~$X$. Let us recall that, if $G$ is a topological subgroup of $\mathrm{Sym}(X)$ for some set $X$, then the subgroups of the form \begin{displaymath}
	U_{G}(F) \, \defeq \, \{ g \in G \mid \forall x \in F \colon \, gx = x \} \qquad (F \in \Pfin (X))
\end{displaymath} constitute a neighborhood basis of the identity in $G$.

\begin{cor}\label{corollary1} Let $X$ be a set. A topological subgroup $G$ of $\mathrm{Sym}(X)$ is skew-amenable if and only if, for every $n \in \mathbb{N}$ and every $x \in X^{n}$, there exists a $G$-invariant mean on~$\ell^{\infty}(Gx)$. \end{cor}

\begin{proof} This is an immediate consequence of Proposition~\ref{proposition:isometry.group}. \end{proof}

If $G$ is a group acting on a set $X$, then we may consider the induced action of $G$ on $\mathscr{P}(X)$ given by $gS \defeq \{ gx \mid x \in S \}$ for all $g \in G$ and $S \in \mathscr{P}(X)$.

\begin{cor}\label{corollary2} Let $(X,{\leq})$ be a linearly ordered set. A topological subgroup $G$ of $\mathrm{Aut}(X,{\leq})$ is skew-amenable if and only if, for every finite subset $S \subseteq X$, there exists a $G$-invariant mean on~$\ell^{\infty}(\{ gS \mid g \in G \})$. \end{cor}

\begin{proof} Since $X$ admits a $G$-invariant linear order, for all $n \in \mathbb{N}$ and $(x_{0},\ldots,x_{n-1}) \in X^{n}$ there is a $G$-equivariant bijection given by \begin{displaymath}
	Gx \, \longrightarrow \, G \{ x_{0},\ldots,x_{n-1} \} , \qquad (y_{0},\ldots,y_{n-1}) \, \longmapsto \, \{ y_{0},\ldots,y_{n-1} \} .
\end{displaymath} Hence, the desired statement is a consequence of Corollary~\ref{corollary1}. \end{proof}
 
\section{The connection with extensive amenability}\label{section:extensive}
 
The purpose of this section if to clarify the connection between skew-amenability of topological groups and the concept of \emph{extensive amenability} of group actions coined by Juschenko, Matte Bon, Monod and de la Salle~\cite{JMMS} based on~\cite{JuschenkoMonod}. We will adopt the categorical approach developed in~\cite{JMMS}.  To this end, let us consider the category $\mathbf{I}$ of finite sets with injective maps as morphisms, as well as the category $\mathbf{Grp}$ of groups with homomorphisms as morphisms. Since the latter category has direct limits, any functor $F \colon \mathbf{I} \to \mathbf{Grp}$ may be extended to a functor $\overline{F}$ from the category of all sets with injective maps to $\mathbf{Grp}$ by setting \begin{displaymath}
	\overline{F}(X) \, \defeq \, \varinjlim\nolimits_{Y \in \Pfin (X)} F(Y)
\end{displaymath} for every set $X$. Appealing to this construction, a functor $F \colon \mathbf{I} \to \mathbf{C}$ is said to be \emph{tight} on a (possibly infinite) set $X$ if, for every\footnote{in case $X$ is non-empty, $F$ is tight on $X$ if and only if this assertion holds for \emph{some} $x \in X$} $x \in X$, the morphism $\overline{F}(\iota) \colon \overline{F}(X \setminus \{ x \}) \to \overline{F}(X)$ induced by the injection $\iota \colon X \setminus \{ x\} \to X, \, y \mapsto y$ is not surjective.
 
\begin{exmpl}\label{example:functor} An important instance of a functor resulting from the described extension construction is the one sending every set $X$ to the group \begin{displaymath}
	H^{(X)} \, \defeq \, \bigoplus\nolimits_{X} H \, = \, \! \left. \left\{ f \in H^{X} \, \right\vert f^{-1}(H\setminus \{ e\}) \text{ finite} \right\} \! ,
\end{displaymath} where $H$ is some fixed group. In this case, tightness of the underlying functor on any non-empty set is equivalent to non-triviality of $H$. \end{exmpl}
 
Let us address some continuity matters related to the construction above.

\begin{lem}\label{lemma:functorial.continuity} Consider a functor $F \colon \mathbf{I} \to \mathbf{Grp}$. Furthermore, let $X$ be a set and let $G$ be a topological subgroup of $\mathrm{Sym}(X)$. Then \begin{displaymath}
	\tau \colon \, G \times \overline{F}(X) \, \longrightarrow \, \overline{F}(X) , \quad (g,f) \, \longmapsto \, \overline{F}(g)(f)
\end{displaymath} is a continuous action of $G$ by automorphisms on the discrete topological group $\overline{F}(X)$. \end{lem}
 
\begin{proof} Evidently, $\tau$ is an action of $G$ by automorphisms on $\overline{F}(X)$. In order to verify continuity, it suffices to show that the associated homomorphism $G \to \mathrm{Sym}(\overline{F}(X)), \, g \mapsto \tau_{g}$ is continuous with respect to the topology of pointwise convergence on $\mathrm{Sym}(X)$ arising from the discrete topology on $\overline{F}(X)$. For this purpose, consider any finite subset $E \subseteq \overline{F}(X)$. By construction of $\overline{F}$, we now find $Y \in \Pfin(X)$ such that $E$ is contained in the image of $\overline{F}(\iota) \colon \overline{F}(Y) \to \overline{F}(X)$ for $\iota \colon Y \to X, \, y \mapsto y$. For each $g \in U_{G}(Y)$, since $g \circ \iota = \iota$, it follows that \begin{displaymath}
	\tau_{g} \circ \overline{F}(\iota) \, = \, \overline{F}(g) \circ \overline{F}(\iota) \, = \, \overline{F}(g \circ \iota) \, = \, \overline{F}(\iota)
\end{displaymath} and thus $\tau_{g}(f) = f$ for all $f \in E$, that is, $\tau_{g} \in U_{\mathrm{Sym}(\overline{F}(X))}(E)$. This completes the proof. \end{proof}
 
\begin{remark}\label{remark:semidirect.product} (1) Let $G$ and $H$ be topological groups and let $\alpha \colon G \times H \to H$ be a continuous action of $G$ by automorphisms on $H$. Then the semi-direct product $H \rtimes_{\alpha} G$ endowed with the product topology and the usual multiplication defined by \begin{displaymath}
	(h,g) \cdot (h',g') \, \defeq \, (h\alpha_{g}(h'),gg') \qquad (h,h' \in H, \, g,g' \in G)
\end{displaymath} constitutes a topological group. Moreover, \begin{displaymath}
	\widetilde{\alpha} \colon \, (H \rtimes_{\alpha} G) \times H \, \longrightarrow \, H , \quad ((h,g),x) \, \longmapsto \, h \alpha_{g}(x)
\end{displaymath} is a continuous action of $H \rtimes_{\alpha} G$ on the topological space $H$.

(2) Consider a functor $F \colon \mathbf{I} \to \mathbf{Grp}$. Moreover, let $X$ be a set and let $G$ be a topological subgroup of $\mathrm{Sym}(X)$. By Lemma~\ref{lemma:functorial.continuity}, the map \begin{displaymath}
	\tau \colon \, G \times \overline{F}(X) \, \longrightarrow \, \overline{F}(X) , \quad (g,f) \, \longmapsto \, \overline{F}(g)(f)
\end{displaymath} is a continuous action of $G$ by automorphisms on the discrete topological group $\overline{F}(X)$. Hence, it follows by~(1) that \begin{displaymath}
	\widetilde{\tau} \colon \, \! \left( \overline{F}(X) \rtimes_{\tau} G\right) \! \times \overline{F}(X) \, \longrightarrow \, \overline{F}(X) , \quad ((f,g),f') \, \longmapsto \, f \tau_{g}(f')
\end{displaymath} is a continuous action of $\overline{F}(X) \rtimes_{\tau} G$ on the discrete space $\overline{F}(X)$.

(3) Let us unravel the above for the functor described in Example~\ref{example:functor} with respect to some fixed group $H$. To this end, let $X$ be a set and $G \leq \mathrm{Sym}(X)$. For any $g \in G$ and $f \in H^{(X)}$, we consider the element $g\boldsymbol{.}f \in H^{(X)}$ defined by \begin{displaymath}
	(g\boldsymbol{.}f)(x) \, \defeq \, f\!\left( g^{-1}x \right) \qquad (x \in X) .
\end{displaymath} Then the action $\tau$ described in Lemma~\ref{lemma:functorial.continuity} is given by $\tau(g,f) = g\boldsymbol{.}f$ for all $g \in G$, $f \in H^{(X)}$. \end{remark}

As is customary, if there is no risk of ambiguity, we may drop the direct reference to the underlying action in a semi-direct product.
 
\begin{definition}[\cite{JMMS}] An action of a group $G$ on a set $X$ is said to be \emph{extensively amenable} if $\ell^{\infty}\!\left( (\mathbb{Z}/2\mathbb{Z})^{(X)} \right)$ admits a $(\mathbb{Z}/2\mathbb{Z})^{(X)} \! \rtimes G$-invariant mean. \end{definition}

Evidently, if an action of a group $G$ on a set $X$ is extensively amenable, then so is the induced action $H \times X \to X , \, (h,x) \mapsto \phi(h)x$ for any group $H$ and any homomorphism $\phi \colon H \to G$. For a variety of non-trivial permanence properties and characterizations of extensive amenability, we refer to~\cite{JMMS}.
 
Now, let us turn to $\mathbf{Amen}$, the full subcategory of $\mathbf{Grp}$ consisting of all amenable groups, and note that $\mathbf{Amen}$ is closed under finite products and direct limits in $\mathbf{Grp}$.
 
\begin{thm}\label{theorem:extensive} Let $X$ be a set and let $G$ be a topological subgroup of $\mathrm{Sym}(X)$. Furthermore, let $F \colon \mathbf{I} \to \mathbf{Amen}$ be a functor. Then the following hold. \begin{itemize}
	\item[$(1)$] If the action of $G$ on $X$ is extensively amenable, then the topological group $\overline{F}(X) \rtimes G$ is skew-amenable.
	\item[$(2)$] Suppose that $F$ is tight on $X$. If the topological group $\overline{F}(X) \rtimes G$ is skew-amenable, then the action of $G$ on $X$ is extensively amenable.
\end{itemize} \end{thm}

\begin{proof} We will keep the notation of Lemma~\ref{lemma:functorial.continuity} and Remark~\ref{remark:semidirect.product}.
	
(1) Assume that the action of $G$ on $X$ is extensively amenable. Let $H \defeq \overline{F}(X) \rtimes_{\tau} G$ and $Y \defeq X \times \overline{F}(X)$. Thanks to~\cite[Theorem~1.3]{JMMS}, the action \begin{displaymath}
	\widetilde{\tau} \colon \, H \times \overline{F}(X) \, \longrightarrow \, \overline{F}(X) , \quad ((f,g),f') \, \longmapsto \, f \tau_{g}(f')
\end{displaymath} is extensively amenable. It now follows by~\cite[Corollary~2.6]{JMMS} that the action \begin{displaymath}
	(G \times H) \times Y \, \longrightarrow \, Y , \quad ((g',(f,g)),(x,f')) \, \longmapsto \, \left( g'x, f \tau_{g}(f') \right)
\end{displaymath} of the group $G \times H$ on the set $Y$ is extensively amenable as well. Upon composition with the homomorphism \begin{displaymath}
	H \, \longrightarrow \, G \times H, \quad (f,g) \, \longmapsto \, (g,(f,g)) ,
\end{displaymath} we obtain an extensively amenable action \begin{displaymath}
	\sigma \colon \, H \times Y \, \longrightarrow \, Y , \quad ((f,g),(x,f')) \, \longmapsto \, \left( gx, f \tau_{g}(f') \right) .
\end{displaymath} According to Remark~\ref{remark:semidirect.product}(2), $\widetilde{\tau}$ is a continuous action of the topological group $H = \overline{F}(X) \rtimes_{\tau} G$ on the discrete topological space $\overline{F}(X)$. Therefore, since $H \to G, \, (f,g) \mapsto g$ is a continuous homomorphism, $\sigma$ constitutes a continuous action of $H$ on the discrete space $Y$, which entails the continuity of the induced homomorphism \begin{displaymath}
	\phi \colon \, H \, \longrightarrow \, \mathrm{Sym}(Y) , \quad (f,g) \, \longmapsto \, \sigma_{(f,g)}
\end{displaymath} with respect to the corresponding topology of pointwise convergence on $\mathrm{Sym}(Y)$. Moreover, $\phi$ is a topological embedding: indeed, if $E \in \Pfin (X)$, then \begin{align*}
	\phi^{-1}\!\left(U_{\mathrm{Sym}(Y)}(E \times \{ e \})\right) \, &= \, \{ (f,g) \in H \mid \forall x \in E \colon \, \phi(f,g)(x,e) = (x,e) \} \\
	&= \, \{ (f,g) \in H \mid \forall x \in E \colon \, (gx,f\tau_{g}(e)) = (x,e) \} \\
	&= \, \{ e \} \times \{ g \in G \mid \forall x \in E \colon \, gx = x \} \\
	&= \, \{ e \} \times U_{G}(E) .
\end{align*} Consequently, in order to prove skew-amenability of $H$, we may appeal to Corollary~\ref{corollary1}. To this end, let $n \in \mathbb{N}$ and $y \in Y^{n}$, and put \begin{displaymath}
	Z \, \defeq \, \{ (\sigma_{h}(y_{1}), \ldots, \sigma_{h}(y_{n})) \mid h \in H \} .
\end{displaymath} Due to~\cite[Corollary~2.6]{JMMS}, extensive amenability of $\sigma$ implies that the induced action \begin{displaymath}
	H^{n} \times Y^{n} \, \longrightarrow \, Y^{n} , \quad ((h_{1},\ldots,h_{n}),(z_{1},\ldots,z_{n})) \, \longmapsto \, (\sigma_{h_{1}}(z_{1}), \ldots, \sigma_{h_{n}}(z_{n}))
\end{displaymath} is extensively amenable. Hence, by~\cite[Corollary~2.3]{JMMS}, there exists a mean on $\ell^{\infty}(Z)$ invariant under the action \begin{displaymath}
	H \times Z \, \longrightarrow \, Z , \quad (h,(z_{1},\ldots,z_{n})) \, \longmapsto \, (\sigma_{h}(z_{1}), \ldots, \sigma_{h}(z_{n})) .
\end{displaymath} By Corollary~\ref{corollary1}, this shows that the topological group $\phi(H) \cong H$ is skew-amenable.

(2) Suppose that the topological group $\overline{F}(X) \rtimes_{\tau} G$ is skew-amenable. Since the action \begin{displaymath}
	\widetilde{\tau} \colon \, \! \left( \overline{F}(X) \rtimes_{\tau} G\right) \! \times \overline{F}(X) \, \longrightarrow \, \overline{F}(X) , \quad ((f,g),f') \, \longmapsto \, f \tau_{g}(f')
\end{displaymath} is continuous with respect to the discrete topology on $\overline{F}(X)$ by Remark~\ref{remark:semidirect.product}(2), it now follows from Corollary~\ref{corollary:discrete.skew.amenability} that $\ell^{\infty}\!\left( \overline{F}(X) \right)$ admits a $\widetilde{\tau}$-invariant mean. Therefore, as $F$ is tight, the action of $G$ on $X$ is extensively amenable by~\cite[Theorem~1.3]{JMMS}. \end{proof}
 
\begin{cor}\label{corollary:extensive} Let $X$ be a set and let $G$ be a topological subgroup of $\mathrm{Sym}(X)$. Furthermore, let $H$ be any non-trivial amenable group and endow the group $H^{(X)}$ with the discrete topology. Then the following are equivalent. \begin{itemize}
	\item[$(1)$] The action of $G$ on $X$ is extensively amenable.
	\item[$(2)$] The topological group $H^{(X)} \! \rtimes G$ is skew-amenable.
\end{itemize} \end{cor}

\begin{proof} This follows by applying Theorem~\ref{theorem:extensive} to Example~\ref{example:functor}. \end{proof}
 
\section{F\o lner-type characterization}\label{section:folner}

This section is devoted to a characterization of skew-amenability in the spirit of F\o lner's amenability criterion~\cite{folner}, continuing the topological approach pursued in~\cite{SchneiderThomFolner,pachl}. The following is a special instance of a more general result by Pachl~\cite{pachl}.

\begin{thm}[cf.~\cite{pachl}, Theorem~3.2]\label{theorem:pachl} A topological group $G$ is skew-amenable if and only if, for every $\epsilon >0$, every $H \in \mathrm{LUEB}(G)$ and every $E \in \Pfin (G)$, there is $\mu \in \Delta (G)$ such that \begin{displaymath}
	\forall g \in E \ \forall f \in H \colon \qquad \vert \mu (f) - \mu (f \circ \lambda_{g}) \vert \, \leq \, \epsilon .
\end{displaymath} \end{thm}

\begin{proof} This is precisely the equivalence of (i) and (iii) in~\cite[Theorem~3.2]{pachl}, evaluated for the left uniformity on $G$. \end{proof}

With Theorem~\ref{theorem:pachl} at hand, we are able to provide a characterization of skew-amenability of topological groups in terms of almost invariant finite subsets: Theorem~\ref{theorem:skew.folner} below. Before proceeding to Theorem~\ref{theorem:skew.folner}, let us note the following auxiliary fact.

\begin{lem}[cf.~\cite{SchneiderThomFolner}, Lemma~2.2]\label{lemma:density} Let $G$ be a non-discrete Hausdorff topological group. Then the set $\{ \delta_{F} \mid F \in \Pfin (G), \, F \ne \emptyset \}$ is dense in $\Delta(G) \subseteq \mathrm{LUEB}(G)^{\ast}$ with respect to the topology of uniform convergence on members of $\mathrm{LUEB}(G)$, i.e., \begin{displaymath}
	\forall \mu \in \Delta (G) \ \forall H \in \mathrm{LUEB}(G) \ \forall \epsilon > 0 \ \exists F \in \Pfin (G) \setminus \{ \emptyset \} \colon \quad \sup\nolimits_{f \in H} \vert \mu (f) - \delta_{F}(f) \vert \leq \epsilon .
\end{displaymath} \end{lem}

\begin{proof} This is an immediate consequence of the more general~\cite[Lemma~2.2]{SchneiderThomFolner}. \end{proof}

We clarify some additional terminology and notation, essentially following~\cite{SchneiderThomFolner}. Consider a \emph{bipartite graph}, that is, a triple $\mathscr{B} = (E,F,R)$ consisting of two finite sets $E$ and $F$ and a subset $R \subseteq E \times F$. If $S \subseteq E$, then $N_{\mathscr{B}}(S) \defeq \{ y \in F \mid \exists x \in S \colon (x,y) \in R \}$. A \emph{matching} in $\mathscr{B}$ is an injective map $\phi \colon D \to F$ such that $D \subseteq E$ and $(x,\phi(x)) \in R$ for all $x \in D$. The \emph{matching number} of $\mathscr{B}$ is defined to be \begin{displaymath}
	\match (\mathscr{B}) \, \defeq \, \sup \{ \vert D \vert \mid \phi \colon D \to F \textnormal{ matching in } \mathscr{B} \} .
\end{displaymath} For the sake of convenience, let us recall a version of Hall's matching theorem~\cite{Hall35} formulated by Ore~\cite[Theorem~8]{Ore}.

\begin{thm}[\cite{Hall35}; \cite{Ore}, Theorem~8]\label{theorem:hall} If $\mathscr{B} = (E,F,R)$ is a bipartite graph, then \begin{displaymath}
	\match (\mathscr{B}) \, = \, |E| - \sup \{ |S| - |N_{\mathscr{B}}(S)| \mid S \subseteq E \} .
\end{displaymath} \end{thm}

In our considerations concerning skew-amenability of topological groups, the following instance of bipartite graphs (and their matching numbers) will be relevant.

\begin{definition} Let $G$ be a topological group. For $E,F \in \Pfin(G)$ and $U \in \mathscr{U}(G)$, we define \begin{displaymath}
	\match_{\Lsh}(E,F,U) \, \defeq \, \match \! \left( E,F, U_{\Lsh} \cap (E \times F) \right) .
\end{displaymath} \end{definition}

The following variation on~\cite[Theorem~4.5]{SchneiderThomFolner} provides a topological variant of F\o lner's theorem~\cite{folner} in the context of skew-amenability.

\begin{thm}\label{theorem:skew.folner} A topological group $G$ is skew-amenable if and only if, for every $\theta \in [0,1)$, every $U \in \mathscr{U}(G)$ and every $E \in \Pfin (G)$, there exists $F \in \Pfin (G)\setminus \{ \emptyset \}$ such that \begin{displaymath}
	\forall g \in E \colon \quad \match_{\Lsh}(F,gF,U) \, \geq \, \theta \vert F \vert .
\end{displaymath} \end{thm}

\begin{proof} The argument is a straightforward adaptation of the proof of~\cite[Theorem~4.5]{SchneiderThomFolner} and will be included only for the sake of completeness. Since each of the two properties of~$G$ (i.e., amenability and the F\o lner-type citerion) is equivalent to the same property of the Hausdorff quotient $G/\bigcap \mathscr{U}(G)$, we may and will assume that $G$ is Hausdorff.
	
($\Longrightarrow$) If $G$ is discrete, then the desired conclusion is due to F\o lner's work~\cite{folner}. Therefore, we may and will assume that $G$ is non-discrete. Let $\theta \in [0,1)$, $U \in \mathscr{U}(G)$ and $E \in \Pfin (G)$. By Urysohn's lemma for uniform space (see, e.g.,~\cite[pp.~182--183]{james}), there exists a left-uniformly continuous function $f \colon G \to [0,1]$ such that $f(e) = 1$ and $f(x) = 0$ for all $x \in G \setminus U$. For every $S \subseteq G$, let \begin{displaymath}
	f_{S} \colon \, G \, \longrightarrow \, [0,1] , \quad x \, \longmapsto \, \sup\nolimits_{s \in S} f\!\left(s^{-1}x\right) .
\end{displaymath} We claim that $H \defeq \{ f_{S} \mid S \subseteq G \}$ is an element of $\mathrm{LUEB}(G)$. Evidently, $H$ is $\Vert \cdot \Vert_{\infty}$-bounded. In order to verify left-uniform equicontinuity, let $\epsilon > 0$. Since $f$ is left-uniformly continuous, there exists $V \in \mathscr{U}(G)$ such that $\vert f(x) - f(y) \vert \leq \epsilon$ for all $(x,y) \in V_{\Lsh}$. In turn, if $(x,y) \in V_{\Lsh}$, then $(gx,gy) \in V_{\Lsh}$ for every $g \in G$, whence \begin{displaymath}
	\left\lvert f_{S}(x) - f_{S}(y) \right\rvert \, = \, \left\lvert \sup\nolimits_{s \in S} f\!\left(s^{-1}x\right) - \sup\nolimits_{s \in S} f\!\left(s^{-1}y\right) \right\rvert \, \leq \, \sup\nolimits_{s \in S} \left\lvert f\!\left(s^{-1}x\right) - f\!\left(s^{-1}y\right) \right\rvert \, \leq \, \epsilon
\end{displaymath} for all $S \subseteq G$. So, $H \in \mathrm{LUEB}(G)$ as claimed. Now, let $\epsilon_{0} \defeq 1-\theta$. According to Theorem~\ref{theorem:pachl}, there exists $\mu \in \Delta (G)$ such that \begin{displaymath}
	\forall g \in E \ \forall h \in H \colon \quad \vert \mu (h \circ \lambda_{g}) - \mu (h) \vert \, \leq \, \tfrac{\epsilon_{0}}{3} .
\end{displaymath} Furthermore, by Lemma~\ref{lemma:density}, there exists $F \in \Pfin (G)\setminus \{ \emptyset \}$ such that $\vert \mu (h) - \delta_{F}(h) \vert \leq \tfrac{\epsilon_{0}}{3}$ for every $h \in H$. Consequently, if $g \in E$ and $S \subseteq G$, then $f_{S} \circ \lambda_{g} = f_{g^{-1}S} \in H$ and thus \begin{displaymath}
	\vert \delta_{F}(f_{S} - f_{S} \circ \lambda_{g}) \vert \leq \vert \delta_{F}(f_{S}) - \mu (f_{S}) \vert + \vert \mu (f_{S}) - \mu (f_{S} \circ \lambda_{g}) \vert + \vert \mu (f_{S} \circ \lambda_{g}) - \delta_{F}(f_{S} \circ \lambda_{g}) \vert \leq \epsilon_{0} .
\end{displaymath} We are going to deduce that $\match_{\Lsh}(F,gF,U) \geq \theta \vert F \vert$ for every $g \in E$. To this end, let $g \in E$ and consider the bipartite graph $\mathscr{B} \defeq (F,gF,U_{\Lsh} \cap (F \times gF))$. Now, if $S \subseteq F$, then \begin{displaymath}
	\vert \delta_{F} (f_{S} \circ \lambda_{g}) - \delta_{F} (f_{S}) \vert \, \leq \, \epsilon_{0} \, = \, 1 - \theta
\end{displaymath} and therefore \begin{align*}
	\vert S \vert \, &\stackrel{S \subseteq f_{S}^{-1}(1)}{\leq} \, \vert F \vert \delta_{F}(f_{S}) \, \leq \, \vert F \vert(1-\theta ) + \vert F \vert \delta_{F} (f_{S} \circ \lambda_{g}) \, = \,  (1-\theta )\vert F \vert + \sum\nolimits_{x \in gF} f_{S}(x) \\
	& \stackrel{\spt (f_{S}) \subseteq SU}{=} \,  (1-\theta )\vert F \vert + \sum\nolimits_{x \in N_{\mathscr{B}}(S)} f_{S}(x) \, \leq \, (1-\theta )\vert F \vert + \vert N_{\mathscr{B}}(S) \vert ,
\end{align*} that is, $\vert S \vert - \vert N_{\mathscr{B}}(S) \vert \leq (1-\theta)\vert F \vert$. We conclude that \begin{displaymath}
	\match_{\Lsh}(F,gF,U) \, \stackrel{\ref{theorem:hall}}{=} \, |F| - \sup \{ |S| - |N_{\mathscr{B}}(S)| \mid S \subseteq F \} \, \geq \, \vert F \vert - (1-\theta) \vert F \vert \, = \, \theta \vert F \vert ,
\end{displaymath} as desired. 

($\Longleftarrow$) We will verify the amenability criterion established in Theorem~\ref{theorem:pachl}. To this end, let $\epsilon \in (0,1]$, $H \in \mathrm{LUEB}(G)$, and $E \in \Pfin (G)$. Consider \begin{displaymath}
	s \, \defeq \, \sup \bigl( \{ \Vert f \Vert_{\infty} \mid f \in H \} \cup \{ 1 \} \bigr) \, \in \, [1,\infty)
\end{displaymath} and let $\theta \defeq 1 - \tfrac{\epsilon}{4s}$. Since $H$ is left-uniformly equicontinuous, we find some $U \in \mathscr{U}(G)$ such that $\vert f(x) - f(y) \vert \leq \tfrac{\epsilon}{2}$ for all $(x,y) \in U_{\Lsh}$ and $f \in H$. Due to our hypothesis, there exists a non-empty finite subset $F \subseteq G$ such that \begin{displaymath}
	\forall g \in E \colon \quad \match_{\Lsh}(F,gF,U) \, \geq \, \theta \vert F \vert .
\end{displaymath} We now claim that $\vert \delta_{F}(f) - \delta_{F}(f \circ \lambda_{g}) \vert \leq \epsilon$ for all $g \in E$ and $f \in H$. In order to prove this, let $g \in E$. Since $\match_{\Lsh}(F,gF,U) \, \geq \, \theta \vert F \vert$, there exists an injective map $\phi \colon D \to gF$ such that $D \subseteq F$, $\vert D \vert \geq \theta \vert F \vert$ and $(x,\phi(x)) \in U_{\Lsh}$ for each $x \in D$. Let us fix any bijection $\psi \colon F \to gF$ with $\psi \vert_{D} = \phi$. Then, for every $f \in H$, we conclude that \begin{align*}
	\vert \delta_{F}(f) - \delta_{F}(f \circ \lambda_{g}) \vert \, & = \, \tfrac{1}{\vert F \vert} \left\lvert \sum\nolimits_{x \in F} f(x) - \sum\nolimits_{x \in F} f(gx) \right\rvert \\
	&= \, \tfrac{1}{\vert F \vert} \left\lvert \left( \sum\nolimits_{x \in D} f(x)-f(\psi (x)) \right) + \left( \sum\nolimits_{x \in F\setminus D} f(x) -f(\psi(x)) \right) \right\rvert \\
	& \leq \, \tfrac{1}{\vert F \vert} \sum\nolimits_{x \in D} \vert f(x)-f(\phi (x)) \vert + \tfrac{1}{\vert F \vert} \sum\nolimits_{x \in F\setminus D} \vert f(x) -f(\psi(x)) \vert \\
	& \leq \, \tfrac{\epsilon \vert D \vert}{2\vert F \vert} + 2s \tfrac{\vert F \vert - \vert D \vert}{\vert F \vert} \, \leq \, \tfrac{\epsilon}{2} + 2s(1-\theta) \, = \, \epsilon . \qedhere
\end{align*} \end{proof}

\section{Persistence properties}\label{section:persistence}
 
We continue by studying persistence properties of skew-amenability. Our first objective is to remark that skew-amenability of topological subgroups is not preserved under taking closures. We illustrate this with the following two obvious examples.

\begin{exmpl} (1) Let $X$ be any infinite set. Then the topological group $\mathrm{Sym}(X)$ is not skew-amenable, due to Corollary~\ref{corollary:discrete.skew.amenability} and the well-known fact that $\ell^{\infty}(X)$ does not admit a $\mathrm{Sym}(X)$-invariant mean. Yet, its dense subgroup \begin{displaymath}
	\{ g \in \mathrm{Sym}(X) \mid \vert \{ x \in X \mid g(x) \ne x \} \vert < \infty \}
\end{displaymath} is locally finite, hence amenable with respect to the discrete topology and therefore skew-amenable with respect to the subspace topology inherited from $\mathrm{Sym}(X)$.

(2) Consider the Hilbert space $\mathscr{H} \defeq \ell^{2}(\mathbb{N})$. The unitary group $U(\mathscr{H})$, equipped with the strong operator topology, is not skew-amenable, as noted by Pestov~\cite[an example after Proposition~3]{Pestov2020}. However, the \emph{Fredholm unitary group} \begin{displaymath}
	U_{C}(\mathscr{H}) \, \defeq \, \{ T \in U(\mathscr{H}) \mid T - \mathrm{Id}_{\mathscr{H}} \text{ compact operator} \}
\end{displaymath} is even extremely amenable with respect to the uniform operator topology, by the work of Gromov and Milman~\cite{GromovMilman}. Since the uniform operator topology is SIN, the group $U_{C}(\mathscr{H})$ is also skew-amenable for this topology, thus skew-amenable for the (coarser) strong operator topology, despite the fact that $U_{C}(\mathscr{H})$ is dense in $U(\mathscr{H})$ with regards to the latter topology. \end{exmpl}

Considering the examples above, the following basic observations appear worth noting.

\begin{lem}\label{lemma:basics} Let $G$ and $H$ be topological groups. Suppose that $G$ is skew-amenable. \begin{itemize}
	\item[$(1)$] If $H$ is a dense subgroup of $G$, then $H$ is skew-amenable.
	\item[$(2)$] If there exists a surjective continuous homomorphism from $G$ onto $H$, then $H$ is skew-amenable.
\end{itemize} \end{lem}

\begin{proof} (1): It follows by Remark~\ref{remark:induced.uniformities}(1) and $H$ being dense in $G$ that the map \begin{displaymath}
	\Psi \colon \, \mathrm{LUCB}(G) \, \longrightarrow \, \mathrm{LUCB}(H), \quad f \, \longmapsto \, f\vert_{H}
\end{displaymath} is an isomorphism of unital real Banach algebras. Moreover, \begin{displaymath}
	\Psi (f) \circ \lambda_{h} \, = \, f\vert_{H} \circ \lambda_{h} \, = \, (f \circ \lambda_{h})\vert_{H} \, = \, \Psi (f \circ \lambda_{h})
\end{displaymath} for all $f \in \mathrm{LUCB}(G)$ and $h \in H$. Therefore, if $\mu \colon \mathrm{LUCB}(G) \to \mathbb{R}$ is a left-invariant mean, then $\mu \circ \Psi^{-1} \colon \mathrm{LUCB}(H) \to \mathbb{R}$ is a left-invariant mean, too.

(2): Let $\phi \colon G \to H$ be a surjective continuous homomorphism. It is straightforward that \begin{displaymath}
	\Phi \colon \, \mathrm{LUCB}(H) \, \longrightarrow \, \mathrm{LUCB}(G), \quad f \, \longmapsto \, f \circ \phi
\end{displaymath} is a well-defined, positive, unital, linear operator. Furthermore, \begin{displaymath}
	\Phi (f) \circ \lambda_{g} \, = \, f \circ \phi \circ \lambda_{g} \, = \, f \circ \lambda_{\phi(g)} \circ \phi \, = \, \Phi \!\left(f \circ \lambda_{\phi(g)}\right)
\end{displaymath} for all $f \in \mathrm{LUCB}(H)$ and $g \in G$. Therefore, since $\phi$ is surjective, if $\mu \colon \mathrm{LUCB}(G) \to \mathbb{R}$ is a left-invariant mean, then the mean $\mu \circ \Phi \colon \mathrm{LUCB}(H) \to \mathbb{R}$ is left-invariant as well. \end{proof}

\begin{lem} Let $N$ be an open, normal subgroup of a topological group $G$. Suppose that the action of $G$ on $N$ by conjugation is uniformly equicontinuous. If $G$ is skew-amenable, then $N$ is skew-amenable. \end{lem}

\begin{proof} Consider the projection $\pi_{N} \colon G \to G/N, \, x \mapsto Nx$. By the axiom of choice, there exists a map $\tau \colon G/N \to G$ such that ${\pi_{N}} \circ {\tau} = {\id_{G/N}}$. For each $f \in \mathbb{R}^{N}$, we define \begin{displaymath}
	Tf \colon \, G \, \longrightarrow \, \mathbb{R}, \quad x \, \longmapsto \, f\!\left( x\tau(Nx)^{-1} \right) .
\end{displaymath} Note that, if $f \in \mathbb{R}^{N}$ and $g \in N$, then \begin{align*}
	T(f \circ \lambda_{g})(x) \, & = \, (f \circ \lambda_{g})\!\left( x\tau(Nx)^{-1} \right) \, = \, f\!\left( gx\tau(Nx)^{-1} \right) \\
	& = \, f\!\left( gx\tau(Ngx)^{-1} \right) \, = \, (Tf)(gx) \, = \, ((Tf) \circ \lambda_{g})(x)
\end{align*} for every $x \in G$. That is, \begin{equation}\tag{1}\label{normal.equivariance}
	\forall f \in \mathbb{R}^{N} \ \forall g \in N \colon \quad T(f \circ \lambda_{g}) \, = \, (Tf) \circ \lambda_{g} .
\end{equation} We now prove that \begin{equation}\tag{2}\label{normal.continuity}
	\forall f \in \mathrm{LUCB}(N) \colon \quad Tf \in \mathrm{LUCB}(G) .
\end{equation} To this end, let $f \in \mathrm{LUCB}(N)$. Evidently, $\Vert Tf \Vert_{\infty} = \Vert f \Vert_{\infty} < \infty$ and therefore $Tf \in \ell^{\infty}(G)$. To show that $Tf \colon G \to \mathbb{R}$ is left-uniformly continuous, consider any $\epsilon > 0$. As $f \in \mathrm{LUCB}(N)$, there exists some $U \in \mathscr{U}(G)$ such that \begin{displaymath}
	\forall x,y \in N \colon \quad x^{-1}y \in U \ \Longrightarrow \ \vert f(x) - f(y) \vert \leq \epsilon .
\end{displaymath} Since the action of $G$ on $N$ by conjugation is uniformly equicontinuous, there exists $V \in \mathscr{U}(G)$ such that \begin{displaymath}
	\forall g \in G \colon \quad g(V \cap N)g^{-1} \subseteq U .
\end{displaymath} Thanks to $N$ being open in $G$, the set $V \cap N$ belongs to $\mathscr{U}(G)$, too. Now, if $x,y \in G$ and $x^{-1}y \in V \cap N$, then in particular $Nx=xN=yN=Ny$, so that furthermore \begin{displaymath}
	\left( x\tau(Nx)^{-1}\right)^{-1}\!\left( y \tau(Ny)^{-1} \right) \, = \, \tau( Nx) x^{-1}y \tau( Ny )^{-1} \, = \, \tau( Nx) x^{-1}y \tau(Nx) \, \in \, U
\end{displaymath} and thus \begin{displaymath}
	\vert (Tf)(x) - (Tf)(y) \vert \, = \, \left\lvert f\!\left( x\tau(Nx)^{-1} \right) - f\!\left( y\tau(Ny)^{-1} \right) \right\rvert \, \leq \, \epsilon .
\end{displaymath} This shows that $Tf \in \mathrm{LUCB}(G)$ and hence proves~\eqref{normal.continuity}. Due to~\eqref{normal.continuity}, the map \begin{displaymath}
	\Phi \colon \, \mathrm{LUCB}(N) \, \longrightarrow \, \mathrm{LUCB}(G) , \quad f \, \longmapsto \, Tf
\end{displaymath} is well defined. It is straightforward to verify that $\Phi$ is linear, positive, and unital. Therefore, if $\mu \colon \mathrm{LUCB}(G) \to \mathbb{R}$ is a left-invariant mean, then the composite $\mu \circ \Phi \colon \mathrm{LUCB}(N) \to \mathbb{R}$ constitutes a mean, which is left-invariant due to~\eqref{normal.equivariance} and left-invariance of $\mu$. Consequently, if $G$ is skew-amenable, then so is $N$. \end{proof}

\begin{lem}\label{lemma:direct.limits} The direct limit of any direct system of skew-amenable topological groups is skew-amenable. \end{lem}

\begin{proof} Consider a direct limit $G = \varinjlim\nolimits_{i \in I} G_{i}$ of skew-amenable topological groups $G_{i}$ $(i \in I)$. For each $i \in I$, the associated continuous homomorphism $\phi_{i} \colon G_{i} \to G$ induces an action \begin{displaymath}
	\alpha_{i} \colon \, G_{i} \times G \, \longrightarrow \, G , \quad (g,x) \, \longmapsto \, \phi_{i}(g) x ,
\end{displaymath} uniformly equicontinuous with respect to the left uniformity of $G$, and thus Proposition~\ref{proposition:skew.amenability} asserts the existence of an $\alpha_{i}$-invariant mean $\mu_{i} \in \mathrm{M}(\mathrm{LUCB}(G))$. As $\mathrm{M}(\mathrm{LUCB}(G))$ is compact with regard to the weak-$\ast$ topology (see, e.g., \cite[2.1, Theorem~1.8(i), p.~68]{AnalysisOnSemigroups}), the net $(\mu_{i})_{i \in I}$ admits a weak-$\ast$ accumulation point $\mu \in \mathrm{M}(\mathrm{LUCB}(G))$. A straightforward argument now shows that $\mu$ is left-invariant: indeed, for any $g \in G$, $f \in \mathrm{LUCB}(G)$ and $\epsilon > 0$, we find $i_{0} \in I$ with $g \in \phi_{i_{0}}(G_{i_{0}})$ and then $i_{1} \in I$ such that $i_{0} \leq i_{1}$ and \begin{displaymath}
	\max \! \left\{ \vert \mu(f) - \mu_{i}(f) \vert, \, \vert \mu(f \circ \lambda_{g}) - \mu_{i}(f \circ \lambda_{g}) \vert \right\} \! \, \leq \, \tfrac{\epsilon}{2} ,
\end{displaymath} which implies that \begin{align*}
	\vert \mu(f \circ \lambda_{g}) - \mu(f) \vert \, &\leq \, \vert \mu(f \circ \lambda_{g}) - \mu_{i}(f \circ \lambda_{g}) \vert + \vert \mu_{i}(f \circ \lambda_{g}) - \mu_{i}(f) \vert + \vert \mu_{i}(f) - \mu(f) \vert \, \leq \, \epsilon . 
\end{align*} Hence, $G$ is skew-amenable. \end{proof}

Our next objective concerns group extensions. Whereas the class of amenable topological groups is closed under extensions~\cite[Theorem~4.8]{rickert}, skew-amenability is not even preserved under semi-direct products in general (see Corollary~\ref{corollary:monod} and Corollary~\ref{corollary:semidirect.monod}). The following definition isolates a sufficient condition for extensions to preserve skew-amenability.

\begin{definition}\label{definition:moderate} A subgroup $N$ of a topological group $G$ is called \emph{moderate} (in $G$) if \begin{displaymath}
	\forall U \in \mathscr{U}(G) \colon \quad \left( \bigcap\nolimits_{g \in N} g^{-1}Ug \right) \! N \in \mathscr{U}(G) .
\end{displaymath} \end{definition}

\begin{remark}\label{remark:moderate} Let $G$ be a topological group and let $N \leq G$. \begin{itemize}
	\item[$(1)$] It is straightforward to verify that $N$ is moderate in $G$ if and only if \begin{displaymath}
					\forall U \in \mathscr{U}(G) \colon \quad N \! \left( \bigcap\nolimits_{g \in N} g^{-1}Ug \right) \in \mathscr{U}(G) .
				\end{displaymath}
	\item[$(2)$] If $N$ is precompact, then a standard argument (see, for instance, \cite[Lemma~3.7.7]{AT}) shows that \begin{displaymath}
					\forall U \in \mathscr{U}(G) \colon \quad \bigcap\nolimits_{g \in N} g^{-1}Ug \, \in \, \mathscr{U}(G) ,
				\end{displaymath} wherefore, in particular, $N$ is moderate in $G$.
	\item[$(3)$] If $N$ is contained in the center of $G$, then $N$ is moderate in $G$.
\end{itemize} \end{remark}

Here is a sufficient criterion for moderateness in the context of semi-direct products.

\begin{lem}\label{lemma:moderate} Let $G$ and $H$ be topological groups and let $\alpha \colon G \times H \to H$ be a bounded action of $G$ by automorphisms on $H$. Then $N \defeq H \times \{ e \}$ is moderate in $H \rtimes_{\alpha} G$. \end{lem} 

\begin{proof} Let $W \in \mathscr{U}(H \rtimes_{\alpha} G)$. Then there exist $U \in \mathscr{U}(G)$ and $V \in \mathscr{U}(H)$ with $V \times U \subseteq W$. Since $\alpha$ is bounded, there exists $U_{0} \in \mathscr{U}(G)$ such that $U_{0} \subseteq U$ and \begin{equation}\tag{$\ast$}\label{bounded}
	\forall u \in U_{0} \ \forall x \in H \colon \quad x \alpha_{u}(x)^{-1} \in V .
\end{equation} It now follows that \begin{equation}\tag{$\ast\ast$}\label{product}
	H \times U_{0} \, \subseteq \, \left( \bigcap\nolimits_{z \in N} z^{-1}Wz \right) \! N .
\end{equation} Indeed, if $(h,u) \in H \times U_{0}$, then $(\alpha_{u^{-1}}(h),e) \in N$, and \begin{align*}
	&(x,e)(h,u)(\alpha_{u^{-1}}(h),e)^{-1}(x,e)^{-1} \! \, = \, (x,e)\!\left(h\alpha_{u}\!\left(\alpha_{u^{-1}}(h)^{-1}\right)\!, u \right)\!(x,e)^{-1} \, = \, (x,e)(e,u)(x,e)^{-1} \\
	& \qquad \qquad \qquad  = \, (x,u)\!\left(x^{-1},e\right) \, = \, \left(x\alpha_{u}\!\left( x^{-1} \right)\!,u \right) \, = \, \left(x\alpha_{u}(x)^{-1},u\right) \, \stackrel{\eqref{bounded}}{\in} \, V \times U_{0} \, \subseteq \, W
\end{align*} for all $x \in H$, that is, \begin{displaymath}
	(h,u)(\alpha_{u^{-1}}(h),e)^{-1} \! \, \in \, \! \left( \bigcap\nolimits_{z \in N} z^{-1}Wz \right) ,
\end{displaymath} thus $(h,u) \in \left( \bigcap\nolimits_{z \in N} z^{-1}Wz \right) \! N$ as desired. Since $H \times U_{0} \in \mathscr{U}(H \rtimes_{\alpha} G)$, assertion~\eqref{product} readily entails that \begin{displaymath}
	\left( \bigcap\nolimits_{z \in N} z^{-1}Wz \right) \! N \in \mathscr{U}(H \rtimes_{\alpha} G) . \qedhere
\end{displaymath}\end{proof}

\begin{cor}\label{corollary:moderate} Let $G$ and $H$ be topological groups. Then $G \times \{ e \}$ is moderate in $G \times H$. \end{cor}

Our interest in moderate subgroups of topological groups is rooted in the following fact.

\begin{prop}\label{proposition:extensions} Let $N$ be a moderate, normal subgroup of a topological group $G$. If both $N$ and $G/N$ are skew-amenable, then $G$ is skew-amenable. \end{prop}

\begin{proof} Let $\mu \colon \mathrm{LUCB}(N) \to \mathbb{R}$ be any left-invariant mean. According to Remark~\ref{remark:induced.uniformities}(1), for each $f \in \mathrm{LUCB}(G)$, a well-defined function $T_{\mu}f \colon G \to \mathbb{R}$ is given by \begin{displaymath}
	(T_{\mu}f)(g) \, \defeq \, \mu ((f \circ \lambda_{g})\vert_{N}) \qquad (g \in G) .
\end{displaymath} First of all, let us note that, for each $f \in \mathrm{LUCB}(G)$, left-invariance of $\mu$ implies that\begin{align*}
	(T_{\mu}f)(gh) \, &= \, \mu ((f \circ \lambda_{gh})\vert_{N}) \, = \, \mu ((f \circ \lambda_{g} \circ \lambda_{h})\vert_{N}) \\
	& = \, \mu ((f \circ \lambda_{g})\vert_{N} \circ \lambda_{h}) \, = \, \mu ((f \circ \lambda_{g})\vert_{N}) \, = \, (T_{\mu}f)(g) \tag{1}\label{invariance}
\end{align*} whenever $g \in G$ and $h \in N$. Next, we will show that \begin{equation}\tag{2}\label{lucb}
	\forall f \in \mathrm{LUCB}(G) \colon \quad T_{\mu}f \in \mathrm{LUCB}(G) .
\end{equation} To this end, let $f \in \mathrm{LUCB}(G)$. Of course, $T_{\mu}f$ is bounded. To prove left-uniform continuity, consider any $\epsilon > 0$. Since $f$ is left-uniformly continuous, there exists $U \in \mathscr{U}(G)$ such that $\Vert f - (f \circ \rho_{g}) \Vert_{\infty} \leq \epsilon$ for all $g \in U$. Thanks to $N$ being moderate, $V \defeq \left( \bigcap\nolimits_{x \in N} x^{-1}Ux\right) \! N$ belongs to $\mathscr{U}(G)$. Now, if $g,h \in G$ and $g^{-1}h \in V$, then we find $u \in \bigcap\nolimits_{x \in N} x^{-1}Ux$ and $z \in N$ with $g^{-1}h = uz^{-1}$, which implies that \begin{displaymath}
	(gx)^{-1}hzx \, = \, x^{-1}g^{-1}hzx \, = \, x^{-1}ux \, \in \, U ,
\end{displaymath} for all $x \in N$, and thus \begin{displaymath}
	\Vert (f \circ \lambda_{g})\vert_{N} - (f \circ \lambda_{hz})\vert_{N} \Vert_{\infty} \, = \, \sup\nolimits_{x \in N} \vert f(gx) - f(hzx) \vert \, \leq \, \epsilon ,
\end{displaymath} which then entails that \begin{displaymath}
	\vert (T_{\mu}f)(g) - (T_{\mu}f)(h) \vert \, = \, \vert (T_{\mu}f)(g) - (T_{\mu}f)(hz) \vert \, = \, \vert \mu ((f \circ \lambda_{g})\vert_{N} - (f \circ \lambda_{hz})\vert_{N}) \vert \, \leq \, \epsilon .
\end{displaymath} This proves~\eqref{lucb}. Thanks to~\eqref{invariance}, for each $f \in \mathrm{LUCB}(G)$, \begin{displaymath}
	S_{\mu}f \colon \, G/N \, \longrightarrow \, \mathbb{R} , \quad gN \, \longmapsto \, (T_{\mu}f)(g)
\end{displaymath} is a well-defined function. Using~\eqref{lucb} and Remark~\ref{remark:induced.uniformities}(2), we furthermore conclude that \begin{displaymath}
	S_{\mu} \colon \, \mathrm{LUCB}(G) \, \longrightarrow \, \mathrm{LUCB}(G/N), \qquad f \, \longmapsto \, S_{\mu}f
\end{displaymath} is a well-defined, positive, unital, linear operator. Also, if $g \in G$ and $f \in \mathrm{LUCB}(G)$, then \begin{align*}
	S_{\mu}(f \circ \lambda_{g})(hN) \, & = \, T_{\mu}(f \circ \lambda_{g})(hN) \, = \, \mu ((f \circ {\lambda_{g}} \circ {\lambda_{h}})\vert_{N}) \, = \, \mu ((f \circ {\lambda_{gh}})\vert_{N}) \\
	& = \, (T_{\mu}f)(gh) \, = \, (S_{\mu}f)(ghN) \, = \, \left({(S_{\mu}f)} \circ {\lambda_{\pi_{N}(g)}}\right)\!(hN)
\end{align*} for all $h \in G$, that is, \begin{equation}\tag{3}\label{equivariance}
	S_{\mu}(f \circ {\lambda_{g}}) \, = \, {(S_{\mu}f)} \circ {\lambda_{\pi_{N}(g)}} .
\end{equation} Finally, picking any left-invariant mean $\nu \colon \mathrm{LUCB}(G/N) \to \mathbb{R}$, we observe that the resulting mean $\xi \defeq \nu \circ {S_{\mu}} \colon \mathrm{LUCB}(G) \to \mathbb{R}$ is also left-invariant: indeed, \begin{displaymath}
	\xi (f \circ \lambda_{g}) \, = \, \nu (S_{\mu}(f \circ \lambda_{g})) \, \stackrel{\eqref{equivariance}}{=} \, \nu \!\left((S_{\mu}f) \circ \lambda_{\pi_{N}(g)}\right) \! \, = \, \nu (S_{\mu}f) \, = \, \xi (f) 
\end{displaymath} for all $f \in \mathrm{LUCB}(G)$ and $g \in G$. Thus, $G$ is skew-amenable. \end{proof}

Taking Remark~\ref{remark:moderate}(3) into account, we see that Proposition~\ref{proposition:extensions} entails the following result by Pestov~\cite[Proposition~16]{Pestov2020}: any central extension of skew-amenable topological groups is skew-amenable. We record some more consequences of Proposition~\ref{proposition:extensions}.

\begin{cor}\label{corollary:semidirect.products}  Let $G$ and $H$ be topological groups and let $\alpha \colon G \times H \to H$ be a bounded action of $G$ by automorphisms on $H$. If both $G$ and $H$ are skew-amenable, then so is $H \rtimes_{\alpha} G$. \end{cor}

\begin{proof} According to Lemma~\ref{lemma:moderate}, the normal subgroup $N \defeq H \times \{ e \}$ is moderate in $H \rtimes_{\alpha} G$. Since the topological groups $N \cong H$ and $(H \rtimes_{\alpha} G)/N \cong G$ are skew-amenable, $H \rtimes_{\alpha} G$ is skew-amenable by Proposition~\ref{proposition:extensions}. \end{proof}

\begin{cor}\label{corollary:precompact.extension} Let $N$ be a precompact normal subgroup of a topological group $G$. If $G/N$ is skew-amenable, then so is $G$. \end{cor}

\begin{proof} Being a precompact topological group, $N$ is amenable by Remark~\ref{remark:precompact.amenable}. Moreover, as noted in Remark~\ref{remark:moderate}(2), precompactness of $N$ entails that \begin{equation}\label{statement}\tag{$\ast$}
	\forall U \in \mathscr{U}(G) \colon \quad \bigcap\nolimits_{g \in N} g^{-1}Ug \, \in \, \mathscr{U}(G) .
\end{equation} It follows from~\eqref{statement} that $N$ has small invariant neighborhoods. Therefore, $N$ is skew-amenable. Furthermore, \eqref{statement} implies that $N$ is moderate in $G$. Since $G/N$ is skew-amenable, too, we may apply Proposition~\ref{proposition:extensions} to conclude that $G$ is skew-amenable. \end{proof}

Of course, Corollary~\ref{corollary:semidirect.products} particularly implies the following.

\begin{cor}\label{corollary:direct.products} If two topological groups $G$ and $H$ are skew-amenable, then so is $G \times H$. \end{cor}

We finish this section by collecting some additional permanence properties.

\begin{prop}\label{proposition:inverse.limits} The inverse limit of any inverse system of skew-amenable topological groups whose limit projections have dense images is skew-amenable. \end{prop}

\begin{proof} Consider an inverse limit \begin{displaymath}
	G \, = \, \varprojlim\nolimits_{i \in I} G_{i}
\end{displaymath} of skew-amenable topological groups $G_{i}$ $(i \in I)$. Suppose that, for each $i \in I$, the image of the corresponding limit projection $\pi_{i} \colon G \to G_{i}$ is dense in $G_{i}$. Then, for every $i \in I$, the map \begin{displaymath}
	\Phi_{i} \colon \, \mathrm{LUCB}(G_{i}) \, \longrightarrow \, \mathrm{LUCB}(G), \quad f \, \longmapsto \, f \circ \pi_{i}
\end{displaymath} is an embedding of unital real Banach algebras satisfying \begin{displaymath}
	\forall g \in G \ \forall f \in \mathrm{LUCB}(G_{i}) \colon \quad \Phi_{i}\!\left(f \circ \lambda_{\pi_{i}(g)}\right) \, = \, \Phi_{i}(f) \circ \lambda_{g} ,
\end{displaymath} whence skew-amenability of $G_{i}$ implies the existence of a left-invariant mean on the subspace $A_{i} \defeq \Phi_{i}(\mathrm{LUCB}(G_{i})) \subseteq \mathrm{LUCB}(G)$. Since $\mathscr{A} \defeq \{ A_{i} \mid i \in I \}$ is upward directed, Lemma~\ref{lemma:directed.union.amenable} thus asserts that the norm-closure $\overline{\bigcup \mathscr{A}} = \mathrm{LUCB}(G)$ admits a left-invariant mean. \end{proof}

\begin{cor}\label{corollary:infinite.direct.products} If $(G_{i})_{i \in I}$ is a family of skew-amenable topological groups, then the direct product $\prod_{i \in I} G_{i}$ is skew-amenable. \end{cor}

\begin{proof} This is an immediate consequence of Corollary~\ref{corollary:direct.products} and Proposition~\ref{proposition:inverse.limits}. \end{proof}

\begin{cor}\label{corollary:infinite.direct.sums} If $(G_{i})_{i \in I}$ is a family of skew-amenable topological groups, then $\bigoplus_{i \in I} G_{i}$ is skew-amenable with respect to the subspace topology inherited from $\prod_{i \in I} G_{i}$. \end{cor}

\begin{proof} This is follows by Corollary~\ref{corollary:infinite.direct.products} and Lemma~\ref{lemma:basics}(1). \end{proof}

\section{Proximal simulation}\label{section:proximal.simulation}

The subject of this section is to establish a feasible sufficient condition for skew-amenability, Proposition~\ref{proposition:weak.proximal.simulation.implies.amenability}, and to deduce some variations and consequences in the context of transformation groups. The key concept for Proposition~\ref{proposition:weak.proximal.simulation.implies.amenability} will be the notion of (\emph{weak}) \emph{proximal simulation}.

\begin{definition}\label{definition:proximal.simulation} Let $G$ be a topological group and consider a subgroup $H \leq G$. We say that \begin{itemize}
	\item[$(1)$\,] $H$ \emph{proximally simulates} $G$ if, for every $E \in \Pfin (G)$, there is $(h_{g})_{g \in E} \in H^{E}$ such that \begin{displaymath}
						\qquad \forall U \in \mathscr{U}(G) \ \forall S \in \Pfin (G) \ \exists t \in G \ \forall g \in E \ \forall s \in S \colon \quad gst \in h_{g}stU .
					\end{displaymath}
	\item[$(2)$\,] $H$ \emph{weakly proximally simulates} $G$ if, for every $E \in \Pfin (G)$, every $F \in \Pfin (\mathrm{LUCB}(G))$ and every $\epsilon > 0$, there exists $(h_{g})_{g \in E} \in H^{E}$ such that \begin{displaymath}
		\qquad \forall S \in \Pfin(H) \ \exists t \in G \ \forall g \in E \ \forall s \in S \ \forall f \in F \colon \quad \vert f(gst) - f(h_{g}st) \vert < \epsilon .
	\end{displaymath}
\end{itemize}\end{definition}

In other words, a topological group $G$ is proximally simulated by a subgroup $H \leq G$ if and only if every finite subset $E \subseteq G$ admits a family $h_{g} \in H$~$(g \in E)$ such that, for every finite subset $S \subseteq G$ and every $U \in \mathscr{U}(G)$, there exists some right-translate $St$ in $G$ such that, restricted to~$St$, each of the left-translations $(\lambda_{g})_{g \in E}$ is $U_{\Lsh}$-close to the corresponding one of $\left(\lambda_{h_{g}}\right)\!_{g \in E}$. Similarly, a topological group $G$ is weakly proximally simulated by a subgroup $H \leq G$ if and only if, for any two finite subsets $E \subseteq G$ and $F \subseteq \mathrm{LUCB}(G)$ and every $\epsilon > 0$, there exists a family of elements $h_{g} \in H$ $(g \in E)$ such that every finite subset $S \subseteq H$ admits some right-translate $St$ in $G$ such that, restricted to $St$, each of the left-translations $(\lambda_{g})_{g \in E}$ is $(F,\epsilon)$-close to the corresponding one of $\left(\lambda_{h_{g}}\right)\!_{g \in E}$.

Evidently, proximal simulation implies weak proximal simulation. Furthermore, let us note that a discrete group $G$ is proximally simulated by a subgroup $H \leq G$ if and only if $H = G$. The statement of Proposition~\ref{proposition:weak.proximal.simulation.implies.amenability} will only require the assumption of weak proximal simulation. However, Lemma~\ref{lemma:approximation.implies.proximal.simulation}---our main source for instances of this phenomenon---provides the stronger conclusion of proximal simulation.

In order to add a dynamical reformulation of Definition~\ref{definition:proximal.simulation}, let us briefly recall some related terminology. To this end, let $G$ be a group acting by isomorphisms on a uniform space $X$. Given an entourage $U$ of $X$, a pair $(x,y) \in X \times X$ will be called \emph{$U$-proximal} with respect to the action of $G$ on $X$ if there exists $g \in G$ such that $(gx,gy) \in U$. As is customary, a pair $(x,y) \in X \times X$ is called \emph{proximal} with respect to the action of $G$ on $X$ if, for every entourage of $X$, the pair $(x,y)$ is $U$-proximal with respect to the action of $G$ on $X$. The following remark connects the dynamical concept of proximality with Definition~\ref{definition:proximal.simulation}.

\begin{remark} Let $G$ be a topological group. For any set $X$, the direct power $G^{X}$ forms a topological group with respect to the associated product topology, and the action of $G$ by right-translations on $G^{X}$, i.e., \begin{displaymath}
	G \times G^{X} \! \, \longrightarrow \, G^{X}, \quad (g,(h_{x})_{x \in X}) \, \longmapsto \, \left(h_{x}g^{-1}\right)_{x \in X} ,
\end{displaymath} constitutes a (bounded) action of $G$ by isomorphisms on the uniform space $G^{X}$ carrying the left uniformity $\mathscr{E}_{\Lsh}\!\left( G^{X}\right)$. Now, define $\alpha \colon G \times G \to G, \, (x,y) \mapsto xy$ and let $H \leq G$. Then \begin{itemize}
	\item[$(1)$\,] $H$ proximally simulates $G$ if and only if, for every $E \in \Pfin (G)$, there exists $\pi \in H^{E}$ such that the pair $(\alpha\vert_{E \times G},\alpha \circ (\pi \times \id))$ is proximal with respect to the action of $G$ by right-translations on the uniform space $\left( G^{E \times G}, \mathscr{E}_{\Lsh}\!\left( G^{E \times G} \right) \right)$.
	\item[$(2)$\,] $H$ weakly proximally simulates $G$ if and only if, for every $E \in \Pfin (G)$, every $\epsilon > 0$ and every $F \in \Pfin (\mathrm{LUCB}(G))$, there exists $\pi \in H^{E}$ such that, for every $S \in \Pfin(H)$, the pair $(\alpha\vert_{E \times G} , \alpha \circ (\pi \times \id ))$ is $V(S,F,\epsilon)$-proximal with respect to the action of $G$ by right-translations on $G^{E \times G}$, where \begin{displaymath}
						\qquad \quad V(S,F,\epsilon) \, \defeq \, \! \left. \left\{ (\phi,\psi) \in G^{E \times G} \, \right\vert \forall z \in E \times S \, \forall f \in F \colon \, \vert f(\phi(z)) - f(\psi(z)) \vert < \epsilon \right\} .
					\end{displaymath}
\end{itemize} \end{remark}

The relevant fact about (weak) proximal simulation that will be used to establish skew-amenability for concrete examples of topological groups in Sections~\ref{section:monod} and~\ref{section:thompson} is the following: every topological group weakly proximally simulated by some discretely amenable subgroup is necessarily skew-amenable. This is precisely the statement of Corollary~\ref{corollary:weak.proximal.simulation.implies.amenability}. To provide a comprehensive picture, we will deduce Corollary~\ref{corollary:weak.proximal.simulation.implies.amenability} from the more general Proposition~\ref{proposition:weak.proximal.simulation.implies.amenability}, the statement of which additionally involves the following piece of terminology.

\begin{definition}\label{definition:balanced.homomorphism} Let $G$ and $H$ be topological groups. A homomorphism $\psi \colon H \to G$ will be called \emph{balanced} if \begin{displaymath}
	\forall U \in \mathscr{U}(G) \colon \quad \bigcap\nolimits_{g \in G} \psi^{-1} \!\left( g^{-1}Ug \right) \! \, \in \, \mathscr{U}(H) .
\end{displaymath} \end{definition}

Of course, every balanced homomorphism between topological groups is necessarily continuous. Moreover, we note that a homomorphism $\psi \colon H \to G$ between topological groups $G$ and $H$ is balanced, for instance, if \begin{itemize}
	\item[---\,] the topology of $H$ is discrete, or
	\item[---\,] $\psi$ is continuous and $G$ is a SIN group, or
	\item[---\,] $H$ is a SIN group and $\psi$ is continuous and surjective.
\end{itemize} Evidently, the identity map from a topological group $G$ into itself is balanced if and only if $G$ is a SIN group.

\begin{lem}\label{lemma:balanced.homomorphism} Let $G$ and $H$ be topological groups, and let $\psi \colon H \to G$ be a balanced homomorphism. For every $f \in \mathrm{LUCB}(G)$, \begin{displaymath}
	\{ f \circ \rho_{g} \circ \psi \mid g \in G \} \, \in \, \mathrm{RUEB}(H) \, .
\end{displaymath} \end{lem}

\begin{proof} Let $f \in \mathrm{LUCB}(G)$. Evidently, $F \defeq \{ f \circ \rho_{g} \circ \psi \mid g \in G \}$ is norm-bounded. In order to prove that $F$ is right-uniformly equicontinuous, let $\epsilon > 0$. Since $f \in \mathrm{LUCB}(G)$, there exists some $U \in \mathscr{U}(G)$ such that \begin{displaymath}
	\forall x,y \in G \colon \qquad x^{-1}y \in U \ \Longrightarrow \ \vert f(x) - f(y) \vert \, \leq \, \epsilon \, .
\end{displaymath} As $\psi$ is balanced, we conclude that \begin{displaymath}
	V \, \defeq \, \bigcap\nolimits_{g \in G} \psi^{-1} \!\left( g^{-1}Ug \right) \! \, \in \, \mathscr{U}(H) \, .
\end{displaymath} Now, if $x,y \in H$ and $yx^{-1} \in V$, then for every $g \in G$ it follows that \begin{displaymath}
	(\psi (x)g)^{-1}\psi(y)g \, = \, (\psi(x)g)^{-1}\psi\!\left(yx^{-1}\right)\! (\psi(x)g) \, \in \, U 
\end{displaymath} and thus \begin{displaymath}
	\vert (f \circ \rho_{g} \circ \psi)(x) - (f \circ \rho_{g} \circ \psi)(y) \vert \, = \, \vert f(\psi(x)g) - f(\psi(y)g) \vert \, \leq \, \epsilon .
\end{displaymath} This shows that $F \in \mathrm{RUEB}(H)$, as desired. \end{proof}

Everything is prepared now to state and prove the above-mentioned Proposition~\ref{proposition:weak.proximal.simulation.implies.amenability}.

\begin{prop}\label{proposition:weak.proximal.simulation.implies.amenability} Let $\psi \colon H \to G$ be a balanced homomorphism between topological groups $G$ and $H$. If $H$ is amenable and $\psi (H)$ weakly proximally simulates $G$, then $G$ is skew-amenable. \end{prop}

\begin{proof} Since $\mathrm{M}(\mathrm{LUCB}(G))$ is compact with respect to the weak-$\ast$ topology (see, e.g., \cite[2.1, Theorem~1.8(i), p.~68]{AnalysisOnSemigroups}), it suffices to prove that \begin{align*}
	\forall \epsilon > 0 \ \forall E \in \Pfin (G) \ \forall F \in \Pfin &(\mathrm{LUCB}(G)) \ \exists \nu \in \Delta (G) \colon \\
	& \sup\nolimits_{g \in E} \sup\nolimits_{f \in F} \vert \nu(f) - \nu (f \circ \lambda_{g}) \vert \, \leq \, \epsilon \, .
\end{align*} To this end, consider any $\epsilon > 0$, $E \in \Pfin (G)$ and $F \in \Pfin (\mathrm{LUCB}(G))$. Since $\psi (H)$ weakly proximally simulates $G$, there exists a family $(h_{g})_{g \in E} \in H^{E}$ such that \begin{equation}\tag{$\ast$}\label{weak.proximal.simulation}
	\forall S \in \Pfin (H) \ \exists t \in G \ \forall g \in E \ \forall s \in S \ \forall f \in F \colon \quad \vert f(g\psi(s)t) - f( \psi(h_{g})\psi(s)t) \vert \leq \tfrac{\epsilon}{2} \, .
\end{equation} Also, by Lemma~\ref{lemma:balanced.homomorphism} and $\psi$ being balanced, \begin{displaymath}
	\{ f \circ \rho_{t} \circ \psi \mid f \in F, \, t \in G \} \, \in \, \mathrm{RUEB}(H) \, .
\end{displaymath} Due to~\cite[Theorem~3.2]{SchneiderThomFolner}, amenability of $H$ implies the existence of $\mu \in \Delta (H)$ such that \begin{displaymath}
	\forall g \in E \ \forall f \in F \ \forall t \in G \colon \quad \left\lvert \mu (f \circ \rho_{t} \circ \psi) - \mu \!\left(f \circ \rho_{t} \circ \psi \circ \lambda_{h_{g}}\right) \right\rvert \, \leq \, \tfrac{\epsilon}{2} \, .
\end{displaymath} Consider $S \defeq \spt (\mu)$. Thanks to~\eqref{weak.proximal.simulation}, there exists $t \in G$ such that \begin{displaymath}
	\forall g \in E \ \forall s \in S \ \forall f \in F \colon \quad \vert f(g\psi(s)t) - f( \psi(h_{g})\psi(s)t) \vert \leq \tfrac{\epsilon}{2} \, .
\end{displaymath} Let $\nu \defeq \sum_{s \in S} \mu (s) \delta_{\psi(s)t} \in \Delta (G)$. By the above, if $g \in E$, then \begin{align*}
	\left\lvert \nu\! \left(f \circ \lambda_{g}\right) - \nu \!\left(f \circ \lambda_{\psi (h_{g})}\right) \right\rvert \, & = \, \left\lvert \sum\nolimits_{s \in S} \mu(s) f(g\psi(s)t) - \sum\nolimits_{s \in S} \mu(s) f(\psi(h_{g})\psi(s)t) \right\rvert \\
	& \leq \, \sum\nolimits_{s \in S} \mu (s) \! \left\lvert f(g\psi(s)t) - f\!\left(\psi(h_{g})\psi(s)t\right) \right\rvert \, \leq \, \tfrac{\epsilon}{2}
\end{align*} for every $f \in F$. Therefore, we conclude that \begin{align*}
	\vert \nu (f) - \nu (f \circ \lambda_{g}) \vert \, &\leq \, \left\lvert \nu (f) - \nu \!\left(f \circ \lambda_{\psi (h_{g})}\right) \right\rvert + \tfrac{\epsilon}{2} \\
	& = \, \left\lvert \mu (f \circ \rho_{t} \circ \psi) - \mu \! \left(f \circ \rho_{t} \circ \psi \circ \lambda_{h_{g}}\right) \right\rvert + \tfrac{\epsilon}{2} \, \leq \, \tfrac{\epsilon}{2} + \tfrac{\epsilon}{2} \, = \, \epsilon
\end{align*} for all $g \in E$ and $f \in F$, as desired. \end{proof}

\begin{cor}\label{corollary:weak.proximal.simulation.implies.amenability} If a topological group $G$ is weakly proximally simulated by some discretely amenable subgroup, then $G$ is skew-amenable. \end{cor}

Before proceeding to further consequences of Proposition~\ref{proposition:weak.proximal.simulation.implies.amenability}, for potential future application we record an analogue of this result concerning (usual) amenability of topological groups. To this end, we will make use of the following standard construction. If $G$ is a topological group, $\mu \in \mathrm{RUCB}(G)^{\ast}$ and $f \in \mathrm{RUCB}(G)$, then the function \begin{displaymath}
	R_{\mu}f \colon \, G \, \longrightarrow \, \mathbb{R}, \quad g \, \longmapsto \, \mu\!\left(f \circ \lambda_{g}\right)
\end{displaymath} belongs to $\mathrm{RUCB}(G)$ (see, e.g.,~\cite[Lemma~3.7(2)]{SchneiderThomRandomWalks} for the proof of a stronger statement). With this additional notation, we formulate the desired counterpart of Proposition~\ref{proposition:weak.proximal.simulation.implies.amenability}.

\begin{prop}\label{proposition:amenable.counterpart} Let $H$ be a subgroup of a topological group $G$ such that \begin{equation}\tag{$\ddagger$}\label{new}
	\begin{split}
	&\forall E \in \Pfin (G) \ \forall F \in \Pfin (\mathrm{RUCB}(G)) \ \forall \epsilon > 0 \ \exists (h_{g})_{g \in E} \in H^{E} \ \forall S \in \Pfin(H) \\
	&\quad \exists \mu \in \mathrm{M}(\mathrm{RUCB}(G)) \ \forall g \in E \ \forall s \in S \ \forall f \in F \colon \quad \vert (R_{\mu}f)(gs) - (R_{\mu}f)(h_{g}s) \vert < \epsilon .
	\end{split}
\end{equation} If the topological subgroup $H$ is amenable, then the topological group $G$ is amenable. \end{prop}

\begin{proof} The proof is a straightforward modification of the proof of Proposition~\ref{proposition:weak.proximal.simulation.implies.amenability}, which we record only for the sake of convenience. Again, thanks to weak-$\ast$ compactness of $\mathrm{M}(\mathrm{RUCB}(G))$ (see, e.g., \cite[2.1, Theorem~1.8(i), p.~68]{AnalysisOnSemigroups}), it suffices to prove that \begin{align*}
	\forall \epsilon > 0 \ \forall E \in \Pfin (G) \ \forall F \in \Pfin &(\mathrm{RUCB}(G)) \ \exists \nu \in \mathrm{M}(\mathrm{RUCB}(G)) \colon \\
		& \sup\nolimits_{g \in E} \sup\nolimits_{f \in F} \vert \nu(f) - \nu (f \circ \lambda_{g}) \vert \, \leq \, \epsilon \, .
\end{align*} For this purpose, let $\epsilon > 0$, $E \in \Pfin (G)$ and $F \in \Pfin (\mathrm{RUCB}(G))$. By our hypothesis~\eqref{new}, there exists a family $(h_{g})_{g \in E} \in H^{E}$ such that \begin{equation}\tag{1}\label{new1} \begin{split}
	\forall S \in \Pfin (H) \ \exists \mu \in \mathrm{M}(\mathrm{RUCB}(G)&) \ \forall g \in E \ \forall s \in S \ \forall f \in F \colon \\
	& \vert (R_{\mu}f)(gs) - (R_{\mu}f)(h_{g}s) \vert \leq \tfrac{\epsilon}{2} \, .
\end{split} \end{equation} By~\cite[Lemma~3.7(2)]{SchneiderThomRandomWalks}, our Remark~\ref{remark:induced.uniformities}(1) and $\mathrm{RUEB}(H)$ being closed under finite unions, \begin{displaymath}
	\{ (R_{\mu}f)\vert_{H} \mid f \in F, \, \mu \in \mathrm{M}(\mathrm{RUCB}(G)) \} \, \in \, \mathrm{RUEB}(H) \, .
\end{displaymath} Hence, due to~\cite[Theorem~3.2]{SchneiderThomFolner} and amenability of $H$, there exists $\xi \in \Delta (H)$ such that \begin{equation}\tag{2}\label{new2}
	\forall g \in E \ \forall f \in F \ \forall \mu \in \mathrm{M}(\mathrm{RUCB}(G)) \colon \quad \left\lvert \xi ((R_{\mu}f)\vert_{H}) - \xi \!\left((R_{\mu}f)\vert_{H} \circ \lambda_{h_{g}}\right) \right\rvert \, \leq \, \tfrac{\epsilon}{2} \, .
\end{equation} Since $S \defeq \spt (\xi)$ is finite, \eqref{new1} asserts the existence of some $\mu \in \mathrm{M}(\mathrm{RUCB}(G))$ such that \begin{equation}\tag{3}\label{new3}
	\forall g \in E \ \forall s \in S \ \forall f \in F \colon \quad \vert (R_{\mu}f)(gs) - (R_{\mu}f)(h_{g}s) \vert \leq \tfrac{\epsilon}{2} \, .
\end{equation} Now, it is straightforward to verify that \begin{displaymath}
	\nu \colon \, \mathrm{RUCB}(G) \, \longrightarrow \, \mathbb{R} , \quad f \, \longmapsto \,  \sum\nolimits_{s \in S} \xi (s) \mu (f \circ \lambda_{s})
\end{displaymath} constitutes a mean on $\mathrm{RUCB}(G)$ satisfying \begin{equation}\tag{4}\label{new4}
	\forall h \in H \ \forall f \in \mathrm{RUCB}(G) \colon \quad \nu(f \circ {\lambda_{h}}) \, = \, \xi ({(R_{\mu}f)\vert_{H}} \circ {\lambda_{h}}) .
\end{equation} By our choice of $\mu$, if $g \in E$, then \begin{align*}
	\left\lvert \nu\! \left(f \circ \lambda_{g}\right) - \nu \!\left(f \circ \lambda_{h_{g}}\right) \right\rvert \, &= \, \left\lvert \sum\nolimits_{s \in S} \xi(s) \mu (f \circ \lambda_{g} \circ \lambda_{s}) - \sum\nolimits_{s \in S} \xi(s) \mu \!\left(f \circ \lambda_{h_{g}} \circ \lambda_{s}\right) \right\rvert \\
		& \hspace{-35mm} \leq \, \sum\nolimits_{s \in S} \xi (s) \! \left\lvert \mu (f \circ \lambda_{gs}) - \mu \!\left(f \circ \lambda_{h_{g}s}\right) \right\rvert \, = \, \sum\nolimits_{s \in S} \xi (s) \! \left\lvert (R_{\mu}f)(gs) - (R_{\mu}f)(h_{g}s) \right\rvert \, \stackrel{\eqref{new3}}{\leq} \, \tfrac{\epsilon}{2}
\end{align*} for every $f \in F$. Consequently, we arrive at \begin{align*}
	\vert \nu (f) - \nu (f \circ \lambda_{g}) \vert \, &\leq \, \left\lvert \nu (f) - \nu \!\left(f \circ \lambda_{h_{g}}\right) \right\rvert + \tfrac{\epsilon}{2} \\
		& \stackrel{\eqref{new4}}{=} \, \left\lvert \xi ( {(R_{\mu}f)\vert_{H}}) - \xi \! \left({(R_{\mu}f)\vert_{H}} \circ \lambda_{h_{g}}\right) \right\rvert + \tfrac{\epsilon}{2} \, \stackrel{\eqref{new2}}{\leq} \, \tfrac{\epsilon}{2} + \tfrac{\epsilon}{2} \, = \, \epsilon
\end{align*} for all $g \in E$ and $f \in F$, as desired. \end{proof}

\begin{remark} Let $G$ be a topological group. \begin{itemize}
	\item[$(1)$] The condition~\eqref{new} in Proposition~\ref{proposition:amenable.counterpart} is trivially satisfied for $H=G$.
	\item[$(2)$] If $G$ is amenable, then~\eqref{new} holds for every subgroup $H \leq G$.
	\item[$(3)$] If~\eqref{new} is true for $H = \{ e \} \leq G$, then $G$ is amenable (not only by Proposition~\ref{proposition:amenable.counterpart}, but by the very statement of~\eqref{new} and weak-$\ast$ compactness of $\mathrm{M}(\mathrm{RUCB}(G))$ already).
\end{itemize} \end{remark}

We return to the study of skew-amenability. Evidently, the usefulness of Proposition~\ref{proposition:weak.proximal.simulation.implies.amenability} and Corollary~\ref{corollary:weak.proximal.simulation.implies.amenability} in concrete applications will depend on the availability of methods for proving weak proximal simulation. In Lemma~\ref{lemma:approximation.implies.proximal.simulation} we provide a sufficient criterion for proximal simulation in the context of topological transformation groups. As a preparation, let us agree on some additional terminology.

\begin{definition}\label{definition:approximation} Let $X$ be a metric space and let $G \leq \mathrm{Iso}(X)$. For any $A \subseteq X$, we let \begin{displaymath}
	B_{X}(A,\epsilon) \, \defeq \, \{ x \in X \mid \exists a \in A \colon \, d_{X}(x,a) < \epsilon \} \, .
\end{displaymath} A subset $A \subseteq X$ will be called \emph{$G$-attractive} if, for every $\epsilon >0$ and every finite subset $F \subseteq X$, there exists some element $g \in G$ such that $g(F) \subseteq B_{X}(A,\epsilon)$. Let $\mathscr{A}$ be a filter on the set $X$. Then $\mathscr{A}$ will be called \begin{enumerate}
	\item[---\,] \emph{$G$-attractive} if every member of $\mathscr{A}$ is $G$-attractive,
	\item[---\,] \emph{approximately $G$-invariant} if $B_{X}(g(A),\epsilon) \in \mathscr{A}$ for all $A \in \mathscr{A}$ and $\epsilon > 0$.
\end{enumerate} Furthermore, a subgroup $H \leq G$ will be said to \emph{$\mathscr{A}$-approximate} $G$ if \begin{displaymath}
	\forall g \in G \ \exists h \in H \ \forall \epsilon > 0 \ \exists A \in \mathscr{A} \colon \quad \sup\nolimits_{x \in A} d_{X}(gx,hx) \, < \, \epsilon \, .
\end{displaymath} \end{definition}

Let us unravel some of the above in the context of topological permutation groups.

\begin{remark}\label{remark:permutation.groups} Let $X$ be a set, and endow $X$ with the (discrete) metric $d$ defined by \begin{displaymath}
	d(x,y) \, \defeq \, \begin{cases}
		\, 0 & \text{if } x=y , \\
		\, 1 & \text{otherwise}
	 \end{cases}
\end{displaymath} for all $x,y \in X$. Consider any subgroup $G \leq \mathrm{Sym}(X) = \mathrm{Iso}(X,d)$. Then a subset $A \subseteq X$ is $G$-attractive if, for every finite subset $F \subseteq X$, there exists some $g \in G$ such that $g(F) \subseteq A$. Let $\mathscr{A}$ be a filter on $X$. Then $\mathscr{A}$ is approximately $G$-invariant if and only if $\mathscr{A}$ is \emph{$G$-invariant}, in the sense that $g(A) \in \mathscr{A}$ for all $A \in \mathscr{A}$. Finally, a subgroup $H \leq G$ $\mathscr{A}$-approximates $G$ if and only if, for every $g \in G$, there exist $h \in H$ and $A \in \mathscr{A}$ such that $g\vert_{A} = h\vert_{A}$. \end{remark}

The following lemma describes the connection between Definition~\ref{definition:proximal.simulation} and Definition~\ref{definition:approximation}.

\begin{lem}\label{lemma:approximation.implies.proximal.simulation} Let $X$ be a metric space, let $G$ be a topological subgroup of $\mathrm{Iso}(X)$, and let $\mathscr{A}$ be a $G$-attractive and approximately $G$-invariant filter on $X$. If $G$ is $\mathscr{A}$-approximated by a subgroup $H \leq G$, then $G$ is proximally simulated by $H$. \end{lem}

\begin{proof} Suppose that $G$ is $\mathscr{A}$-approximated by some $H \leq G$. To show that $H$ proximally simulates $G$, consider a finite subset $E \subseteq G$. Since $G$ is $\mathscr{A}$-approximated by $H$, there exists a family $(h_{g})_{g \in E} \in H^{E}$ such that \begin{equation}\tag{$\ast$}\label{hypothesis}
	\forall \epsilon > 0 \ \exists (A_{g})_{g \in E} \in \mathscr{A}^{E} \ \forall g \in E \colon \quad \sup\nolimits_{x \in A_{g}} d_{X}(gx,h_{g}x) \, < \, \epsilon \, .
\end{equation} We now claim that \begin{equation}\tag{$\ast\ast$}\label{claim}
	\forall U \in \mathscr{U}(G) \ \forall S \in \Pfin (G) \ \exists t \in G \ \forall g \in E \ \forall s \in S \colon \quad gst \in h_{g}stU \, .
\end{equation} In order to prove~\eqref{claim}, let $U \in \mathscr{U}(G)$ and $S \in \Pfin(G)$. By definition of the topology of pointwise convergence, there exist some finite subset $F \subseteq X$ and some $\epsilon > 0$ such that $U_{G}(F,\epsilon) \subseteq U$. Using assertion~\eqref{hypothesis}, we find $(A_{g})_{g \in E} \in \mathscr{A}^{E}$ such that $\sup\nolimits_{x \in A_{g}} d_{X}(gx,h_{g}x) < \tfrac{\epsilon}{3}$ for each $g \in E$. Furthermore, thanks to $\mathscr{A}$ being approximately $G$-invariant, \begin{displaymath}
	A \, \defeq \, \bigcap\nolimits_{s \in S} \bigcap\nolimits_{g \in E} B_{X}\!\left(s^{-1}(A_{g}),\tfrac{\epsilon}{6}\right) \! \, \in \, \mathscr{A} \, .
\end{displaymath} In particular, $A$ is $G$-attractive, and so there exists some $t \in G$ such that $t(F) \subseteq B_{X}\!\left(A,\tfrac{\epsilon}{6}\right)$. Hence, if $g \in E$, $s \in S$ and $x \in F$, then $tx \in B_{X}\!\left( s^{-1}(A_{g}),\tfrac{\epsilon}{3} \right)$, i.e., $stx \in B_{X}\!\left(A_{g},\tfrac{\epsilon}{3}\right)$, so that we find $y \in A_{g}$ with $d_{X}(stx,y) < \tfrac{\epsilon}{3}$, wherefore \begin{displaymath}
	d_{X}\!\left(gstx,h_{g}st x)\right) \, \leq \, d_{X}(gstx, gy) + d_{X}(gy, h_{g}y) + d_{X}(h_{g}y, h_{g}stx) \, < \, \epsilon \, .
\end{displaymath} So, for all $g \in E$ and $s \in S$, we have $\left( h_{g}st \right)^{-1}\!gst \in U_{G}(F,\epsilon) \subseteq U$, and thus $gst \in h_{g}stU$. This proves~\eqref{claim}, as desired. \end{proof}

We arrive at the following skew-amenability criterion for topological isometry groups.

\begin{prop}\label{proposition:approximation.implies.skew.amenability} Let $X$ be a metric space, let $G$ be a topological subgroup of $\mathrm{Iso}(X)$, and let $\mathscr{A}$ be a $G$-attractive and approximately $G$-invariant filter on $X$. Let $H$ be an amenable topological group and let $\psi \colon H \to G$ be a balanced homomorphism. If $G$ is $\mathscr{A}$-approximated by $\psi(H)$, then $G$ is skew-amenable. \end{prop}

\begin{proof} This is an immediate consequence of Lemma~\ref{lemma:approximation.implies.proximal.simulation} and Proposition~\ref{proposition:weak.proximal.simulation.implies.amenability}. \end{proof}

\begin{cor}\label{corollary:approximation.implies.skew.amenability} Let $X$ be a metric space, let $G$ be a topological subgroup of $\mathrm{Iso}(X)$, and let $\mathscr{A}$ be a $G$-attractive and approximately $G$-invariant filter on $X$. If $G$ is $\mathscr{A}$-approximated by some discretely amenable subgroup, then $G$ is skew-amenable. \end{cor}

We conclude this section with discrete versions of Lemma~\ref{lemma:approximation.implies.proximal.simulation} and Proposition~\ref{proposition:approximation.implies.skew.amenability} in the context of permutation groups, viewed as isometry groups (cf.~Remark~\ref{remark:permutation.groups}).

\begin{cor}\label{corollary:permutation.groups} Let $X$ be a set, let $G$ be a topological subgroup of $\mathrm{Sym}(X)$, and let $\mathscr{A}$ be a $G$-attractive and $G$-invariant filter on $X$. Then the following hold: \begin{itemize}
	\item[$(1)$] If $G$ is $\mathscr{A}$-approximated by some $H \leq G$, then $G$ is proximally simulated by $H$.
	\item[$(2)$] If there exists an amenable topological group $H$ along with a balanced homomorphism $\psi \colon H \to G$ such that $G$ is $\mathscr{A}$-approximated by $\psi(H)$, then $G$ is skew-amenable.
	\item[$(3)$] If $G$ is $\mathscr{A}$-approximated by a discretely amenable subgroup, then $G$ is skew-amenable.
\end{itemize} \end{cor}

\section{Example: Monod's group}\label{section:monod}

This section is devoted to a detailed discussion of skew-amenability and related properties of Monod's groups of piecewise projective homeomorphisms of the real line~\cite{monod}. We recall Monod's construction~\cite{monod} briefly. Let $A$ be a unital subring of $\mathbb{R}$ and consider the natural faithful action of the associated projective special linear group $\mathrm{PSL}(2,A) = \mathrm{SL}(2,A)/\{\pm E_{2}\}$ by homeomorphisms on the real projective line $\mathbf{P}^{1} \defeq \mathbf{P}^{1}(\mathbb{R})$, the latter being equipped with the \emph{projective topology}, i.e., the quotient topology inherited from the unit sphere $S^{1} \subseteq \mathbb{R}^{2}$. Henceforth, we will identify $\mathrm{PSL}(2,A)$ with the corresponding subgroup of $\mathrm{Homeo}\!\left( \mathbf{P}^{1} \right)$. Let us denote by $H(A)$ the group of such homeomorphisms of $\mathbf{P}^{1}$ which fix $\infty = [1:0] \in \mathbf{P}^{1}$ and are piecewise in $\mathrm{PSL}(2,A)$, with finitely many pieces, each of which being an interval in the cyclic order of $\mathbf{P}^{1}$ and having endpoints contained in the set \begin{displaymath}
	\left. P_{A} \, \defeq \, \! \left\{ z \in \mathbf{P}^{1} \, \right\vert \exists h \in \mathrm{PSL}(2,A) \text{ hyperbolic: }  hz = z \right\} .
\end{displaymath} By the work of Monod~\cite[Theorem~1]{monod}, if $A \ne \mathbb{Z}$, then the discrete group $H(A)$ is not amenable. In the present note, we study amenability-like properties of $H(A)$ viewed as a topological group. For this purpose, we observe that there are (at least) two natural non-discrete group topologies on $H(A)$: we denote by $H(A)_{p}$ the topological group $H(A)$ endowed with the topology of pointwise convergence arising from the discrete topology on~$\mathbf{P}^{1}$, and we denote by $H(A)_{u}$ the topological group $H(A)$ endowed with the topology of uniform convergence arising from the projective topology on $\mathbf{P}^{1}$. Since $\mathbf{P}^{1}$ is homeomorphic to $S^{1}$, the topology
of uniform convergence on $\mathrm{Homeo}\!\left( \mathbf{P}^{1} \right)$ coincides\footnote{An interesting discussion of properties of a topological space $X$ ensuring that the associated compact-open topology and the topology of pointwise convergence agree on $\mathrm{Homeo}(X)$ is to be found in~\cite[Remark~3, footnote~2, pp.~3--4]{gheysens}.} with the topology of pointwise convergence associated with the projective topology on $\mathbf{P}^{1}$, which is coarser than the topology of pointwise convergence corresponding to the discrete
topology on $\mathbf{P}^{1}$, of course. In particular, the topology of $H(A)_{u}$ is coarser than the one carried by $H(A)_{p}$.

\begin{remark}\label{remark:monod} Let $A$ be a unital subring of $\mathbb{R}$ and denote by $A^{\ast}$ the group of its multiplicative units. Then the subgroup \begin{displaymath}
	\left. U(A) \, \defeq \, \left\{ \begin{pmatrix} a & b \\ 0 & a^{-1} \end{pmatrix} \, \right\vert a \in A^{\ast}, \, b \in A \right\} \, \leq \, \mathrm{SL}(2,A)
\end{displaymath} is amenable, being the semi-direct product of the two abelian subgroups \begin{displaymath}
	\left. \left\{ \begin{pmatrix} a & 0 \\ 0 & a^{-1} \end{pmatrix} \, \right\vert a \in A^{\ast} \right\} \, \leq \, U(A) \, , \qquad \left. \left\{ \begin{pmatrix} 1 & b \\ 0 & 1 \end{pmatrix} \, \right\vert b \in A \right\} \, \unlhd\, U(A) \, .
\end{displaymath} In turn, the quotient group $T(A) \defeq U(A)/\{ \pm E_{2} \} \leq \mathrm{PSL}(2,A)$ is amenable, too. \end{remark}

\begin{prop}\label{proposition:monod} Let $A$ be a unital subring of $\mathbb{R}$. The topological group $H(A)_{p}$ is proximally simulated by its subgroup $T(A) \leq H(A)$. \end{prop}

\begin{proof} Let us consider the filter \begin{displaymath}
	\left. \mathscr{A} \defeq \left\{ B \subseteq \mathbf{P}^{1}  \, \right\vert \exists a \in \mathbb{R} \colon \, \{ [b:1] \mid b \in \mathbb{R}, \, a \leq b \} \cup \{ \infty \} \subseteq B \right\} .
\end{displaymath} Since $\infty \in \mathbf{P}^{1}$ is an $H(A)$-fixed point and $H(A)$ preserves the cyclic order of $\mathbf{P}^{1}$, the filter $\mathscr{A}$ is $H(A)$-invariant. Furthermore, an argument by Monod~\cite[Lemma~16]{monod} shows that $\mathscr{A}$ is $H(A)$-attractive. Finally, as $T(A)$ is precisely the stabilizer of $\infty = [1:0] \in \mathbf{P}^{1}$ in $\mathrm{PSL}(2,A)$, the definition of $H(A)$ readily entails that $H(A)$ is $\mathscr{A}$-approximated by $T(A)$, whence $H(A)$ is proximally simulated by $T(A)$, due to Corollary~\ref{corollary:permutation.groups}(1). \end{proof}

\begin{cor}\label{corollary:monod} Let $A$ be a unital subring of $\mathbb{R}$. Then $H(A)_{p}$ is skew-amenable. \end{cor}

\begin{proof} This is a consequence of Proposition~\ref{proposition:monod}, Remark~\ref{remark:monod}, and Corollary~\ref{corollary:weak.proximal.simulation.implies.amenability}. \end{proof}

Alternatively, Corollary~\ref{corollary:monod} above may be proved using Corollary~\ref{corollary1} and Remark~\ref{remark:monod}, via a more concrete rendering of the argument proving Proposition~\ref{proposition:weak.proximal.simulation.implies.amenability}. Similarly, one may deduce Corollary~\ref{corollary:F} below from Corollary~\ref{corollary1} and amenability of $\mathbb{Z}$, by suitably reproducing the argument from the proof of Proposition~\ref{proposition:weak.proximal.simulation.implies.amenability}. Furthermore, let us note the following immediate consequences of our observations above:

\begin{cor} Let $A$ be a unital subring of $\mathbb{R}$. The following hold: \begin{itemize}
	\item[$(1)$] \,$H(A)_{u}$ is proximally simulated by its subgroup $T(A) \leq H(A)$.
	\item[$(2)$] \,$H(A)_{u}$ is skew-amenable.
\end{itemize} \end{cor}

\begin{proof} Since the topology of $H(A)_{p}$ is finer than that of $H(A)_{u}$, statement~(1) is an immediate consequence of Proposition~\ref{proposition:monod}, and~(2) follows directly from Corollary~\ref{corollary:monod}. \end{proof}

On the other hand, combining Corollary~\ref{corollary:extensive} with the work Juschenko, Matte Bon, Monod, and de la Salle~\cite{JMMS}, we make the following observation.

\begin{cor}\label{corollary:semidirect.monod} Let $A$ be any non-trivial amenable group and equip the group $A^{(\mathbf{P}^{1})}$ with the discrete topology. Then $A^{(\mathbf{P}^{1})} \! \rtimes H(\mathbb{R})_{p}$ is not skew-amenable. \end{cor}

\begin{proof} Due to~\cite[Theorem~6.1]{JMMS}, the action of $H(\mathbb{R})$ on $\mathbf{P}^{1}$ is not extensively amenable. Consequently, the desired statement follows by Corollary~\ref{corollary:extensive}. \end{proof}

In particular, Corollary~\ref{corollary:semidirect.monod} provides an example of a skew-amenable (non-archimedean) topological group acting continuously by automorphisms on a discrete abelian group such that the resulting semi-direct product is not skew-amenable.

We conclude this section by remarking that the topological group $H(\mathbb{R})_{p}$ is furthermore extremely amenable, as follows from Pestov's work~\cite{pestov}. To provide some more detail, let us endow the set $\mathbf{P}^{1}\setminus \{ \infty\}$ with the linear order $\preceq$ inherited from $\mathbb{R}$ under the bijection \begin{displaymath}
	\mathbb{R} \, \longrightarrow \, \mathbf{P}^{1}\setminus \{ \infty\}, \quad x \, \longmapsto \, [x:1] ,
\end{displaymath} i.e., the one defined by \begin{displaymath}
	[x:1] \preceq [y:1] \quad :\Longleftrightarrow \quad x \leq y \qquad (x,y \in \mathbb{R}),
\end{displaymath} and note that $\preceq$ is $H(\mathbb{R})$-invariant. The following observations are well known.

\begin{lem}\label{lemma:two.transitive} Let $x_{0},x_{1},y_{0},y_{1} \in \mathbf{P}^{1}\setminus \{ \infty\}$ such that $x_{0} \prec x_{1}$ and $y_{0} \prec y_{1}$. Then there exists $g \in \mathrm{PSL}(2,\mathbb{R})$ such that $gx_{0}=y_{0}$ and $gx_{1}=y_{1}$. \end{lem}

\begin{proof} Let $s_{0},s_{1},t_{0},t_{1} \in \mathbb{R}$ such that $x_{0} = [s_{0}:1]$, $x_{1} = [s_{1}:1]$, $y_{0} = [t_{0}:1]$, and $y_{1} = [t_{1}:1]$. Since $x_{0} \prec x_{1}$ and $y_{0} \prec y_{1}$, we have $s_{0} < s_{1}$ and $t_{0} < t_{1}$. Hence, considering \begin{displaymath}
	A \, \defeq \, \begin{pmatrix} t_{0} & t_{1} \\ 1 & 1 \end{pmatrix} \begin{pmatrix} s_{0} & s_{1} \\ 1 & 1 \end{pmatrix}^{-1} \in \, \mathrm{GL}(2,\mathbb{R}) ,
\end{displaymath} we conclude that \begin{displaymath}
	d \, \defeq \, \det A \, = \, \tfrac{t_{0}-t_{1}}{s_{0}-s_{1}} \, > \, 0 .
\end{displaymath} Consequently, the matrix $\tfrac{1}{\sqrt{d}}A$ belongs to $\mathrm{SL}(2,\mathbb{R})$, thus projects to an element $g \in \mathrm{PSL}(2,\mathbb{R})$. By construction, \begin{displaymath}
	A \begin{pmatrix} s_{0} \\ 1 \end{pmatrix} = \begin{pmatrix} t_{0} \\ 1 \end{pmatrix} , \qquad A \begin{pmatrix} s_{1} \\ 1 \end{pmatrix} = \begin{pmatrix} t_{1} \\ 1 \end{pmatrix}
\end{displaymath} and therefore $gx_{0}=y_{0}$ and $gx_{1}=y_{1}$, as desired. \end{proof}

\begin{lem}\label{lemma:one.transitive} For any $x,y \in \mathbf{P}^{1}\setminus \{ \infty\}$, there exists $g \in \mathrm{PSL}(2,\mathbb{R})$ with $gx=y$ and $g\infty = \infty$. \end{lem}

\begin{proof} Let $s,t \in \mathbb{R}$ and consider $x \defeq [s:1], \, y \defeq [t:1] \in \mathbf{P}^{1}\setminus \{ \infty\}$. Then the matrix \begin{displaymath}
	A \, \defeq \, \begin{pmatrix} 1 & t-s \\ 0 & 1\end{pmatrix} \in \, \mathrm{SL}(2,\mathbb{R})
\end{displaymath} satisfies \begin{displaymath}
	A \begin{pmatrix} s \\ 1 \end{pmatrix} = \begin{pmatrix} t \\ 1 \end{pmatrix} , \qquad A \begin{pmatrix} 1 \\ 0 \end{pmatrix} = \begin{pmatrix} 1 \\ 0 \end{pmatrix} ,
\end{displaymath} and hence projects to an element $g \in \mathrm{PSL}(2,\mathbb{R})$ satisfying $gx=y$ and $g\infty = \infty$. \end{proof}

As usual, an action of a group $G$ on a set $X$ is called \emph{strongly transitive} if, for every $n \in \mathbb{N}$, the induced action of $G$ on the set of all $n$-element subsets of $X$ is transitive.

\begin{prop}\label{proposition:strong.transitivity} The action of $H(\mathbb{R})$ on $\mathbf{P}^{1}\setminus \{ \infty\}$ is strongly transitive. \end{prop}

\begin{proof} Consider any two sequences $x_{0},\ldots,x_{n-1},y_{0},\ldots,y_{n-1} \in \mathbf{P}^{1}\setminus \{ \infty\}$ with $x_{0} \prec \ldots \prec x_{n-1}$ and $y_{0} \prec \ldots \prec y_{n-1}$. By Lemma~\ref{lemma:two.transitive}, we find $g_{1},\ldots,g_{n-1} \in \mathrm{PSL}(2,\mathbb{R})$ such that \begin{displaymath}
	\forall i \in \{ 1,\ldots,n-1 \} \colon \qquad g_{i}x_{i-1} = y_{i-1}, \quad g_{i}x_{i} = y_{i} .
\end{displaymath} Moreover, by Lemma~\ref{lemma:one.transitive}, there exist $g_{0},g_{n} \in \mathrm{PSL}(2,\mathbb{R})$ such that $g_{0}x_{0} = y_{0}$, $g_{n}x_{n-1} = y_{n-1}$ and $g_{0}\infty = \infty = g_{n}\infty$. Furthermore, we observe that $P_{\mathbb{R}} = \mathbf{P}^{1}$. Thus, we may define an element $g \in H(\mathbb{R})$ by setting $g \infty \defeq \infty$ and \begin{displaymath}
	gx \, \defeq \, \begin{cases}
						\, g_{0}x & \text{if } x \preceq x_{0} , \\
						\, g_{1}x & \text{if } x_{0} \prec x \preceq x_{1} , \\
						\, \vdots & \vdots  \\
						\, g_{n-1}x & \text{if } x_{n-2} \prec x \preceq x_{n-1} , \\
						\, g_{n}x & \text{if } x_{n-1} \prec x
					\end{cases}
\end{displaymath} for every $x \in \mathbf{P}^{1}\setminus \{ \infty\}$. Evidently, $gx_{i} = y_{i}$ for each $i \in \{ 0,\ldots,n-1 \}$, as desired. \end{proof}

As a consequence of Proposition~\ref{proposition:strong.transitivity}, the topological group $H(\mathbb{R})_{p}$ is isomorphic to a dense topological subgroup of $\mathrm{Aut}\! \left( \mathbf{P}^{1} \setminus \{ \infty\}, {\preceq}\right)$. Since the latter is extremely amenable by Pestov's work~\cite{pestov}, we arrive at the following corollary.

\begin{cor}\label{corollary:monod.extremely.amenable} $H(\mathbb{R})_{p}$ is extremely amenable. \end{cor}

\begin{proof} Let $X \defeq \mathbf{P}^{1} \setminus \{ \infty\}$. Since $H(\mathbb{R})$ fixes $\infty \in \mathbf{P}^{1}$ and preserves $\preceq$, the map \begin{displaymath}
	H(\mathbb{R})_{p} \, \longrightarrow \, \mathrm{Aut}(X,\preceq), \quad g \, \longmapsto \, g\vert_{X}
\end{displaymath} constitutes an embedding of topological groups. By Proposition~\ref{proposition:strong.transitivity}, the action of $H(\mathbb{R})$ on $X$ is strongly transitive, whence $H(\mathbb{R})_{p}$ is extremely amenable by~\cite[Theorem~5.4]{pestov}. \end{proof}

As the topology of $H(\mathbb{R})_{p}$ is finer than that of $H(\mathbb{R})_{u}$, the above Corollary~\ref{corollary:monod.extremely.amenable} immediately entails the extreme amenability of the topological group $H(\mathbb{R})_{u}$.

\section{Example: Thompson's group $F$}\label{section:thompson}

For this section, let us turn to Richard Thompson's group~$F$, defined to be the subgroup of $\mathrm{Aut} \!\left( \mathbb{Z}\! \left[ \tfrac{1}{2} \right],{\leq} \right)$ consisting precisely of those automorphisms which are piecewise affine (with finitely many pieces), have all their slopes contained in $\! \left. \left\{ 2^{k} \, \right\vert k \in \mathbb{Z} \right\}$, and furthermore agree with translations by (possibly two different) integers on $(-\infty, a] \cap \mathbb{Z}\! \left[ \tfrac{1}{2} \right]$ and $[a,\infty) \cap \mathbb{Z}\! \left[ \tfrac{1}{2} \right]$ for a sufficiently large dyadic rational $a \in \mathbb{Z}\!\left[ \tfrac{1}{2} \right]$. For more comprehensive background material on Thompson's group~$F$, the reader is referred to~\cite{BrinSquier,CannonFloydParry,haagerup,kaimanovich}. In the following, $\mathrm{Aut} \!\left( \mathbb{Z}\! \left[ \tfrac{1}{2} \right],{\leq} \right)$ and its subgroup $F$ will be considered as topological groups, carrying the respective topology of pointwise convergence induced by the action on the discrete topological space $\mathbb{Z}\! \left[ \tfrac{1}{2} \right]$, i.e., the subspace topology inherited from $\mathrm{Sym}\!\left( \mathbb{Z}\! \left[ \tfrac{1}{2} \right] \right)$.

\begin{prop}\label{proposition:F} The topological group $F$, carrying the topology of pointwise convergence, is proximally simulated by its subgroup $H \cong \mathbb{Z}$ consisting of all integer translations of $\mathbb{Z}\!\left[ \tfrac{1}{2} \right]$. \end{prop}

\begin{proof} We observe that the filter \begin{displaymath}
	\mathscr{A} \, \defeq \, \left\{ A \subseteq \mathbb{Z}\!\left[ \tfrac{1}{2} \right]  \left\vert \, \exists a \in \mathbb{Z}\!\left[ \tfrac{1}{2} \right] \! \colon \, (-\infty, a] \cap \mathbb{Z}\! \left[ \tfrac{1}{2} \right] \subseteq A \right\} \right.
\end{displaymath} is both $F$-invariant and $F$-attractive. The definition of $F$ implies that $H$ $\mathscr{A}$-approximates $F$. By Corollary~\ref{corollary:permutation.groups}(1), this readily entails that $F$ is proximally simulated by $H$. \end{proof}

\begin{cor}\label{corollary:F} $F$ is skew-amenable with respect to the topology of pointwise convergence. \end{cor}

\begin{proof} This follows directly from Proposition~\ref{proposition:F}, Corollary~\ref{corollary:weak.proximal.simulation.implies.amenability}, and amenability of $\mathbb{Z}$. \end{proof}

It is well known that the action of $F$ on $\mathbb{Z}\! \left[ \tfrac{1}{2} \right]$ is strongly transitive (see~\cite[Lemma~4.2]{CannonFloydParry}, for instance), which entails that $F$ constitutes a dense subgroup of $\mathrm{Aut} \!\left( \mathbb{Z}\! \left[ \tfrac{1}{2} \right],{\leq} \right)$. Hence, this gives yet another example showing that skew-amenability is not preserved under topological closures in general: despite skew-amenability of $F$ with respect to the topology of pointwise convergence, $\mathrm{Aut} \!\left( \mathbb{Z}\! \left[ \tfrac{1}{2} \right],{\leq} \right)$ fails to be skew-amenable (Proposition~\ref{proposition:aut.is.not.skew.amenable}).

Preparing the proof of Proposition~\ref{proposition:aut.is.not.skew.amenable}, we briefly recall some background concerning orderability of groups. A linear order $\preceq$ on a set $X$ will be called \emph{dense} if \begin{displaymath}
	\forall x,y \in X \colon \qquad (x \, \preceq \, y \ \wedge \ x \, \ne \, y) \ \Longrightarrow \  \left(\exists z \in X \setminus \{ x,y \} \colon \ x \, \preceq \, z \, \preceq \, y \right) .
\end{displaymath} Now, let $G$ be a group. A partial order $\preceq$ on $G$ is called \emph{left-invariant} (resp.,~\emph{right-invariant}) if $gx \preceq gy$ (resp.,~$xg \preceq yg$) for all $g,x,y \in G$ with $x \preceq y$. A partial order on~$G$ is said to be \emph{invariant} if it is both left- and right-invariant. Following~\cite{LinnellRhemtullaRolfsen}, a left-invariant (or right-invariant) linear order $\preceq$ on $G$ will be called \emph{discrete} if \begin{displaymath}
	\exists g \in G \setminus \{ e\} \colon \quad \{ h \in G \mid e \preceq h \preceq g \} \, = \, \{ e,g \} .
\end{displaymath} As is customary, $G$ will be called \emph{left-orderable} (resp.,~\emph{orderable}) if $G$ admits a left-invariant (resp.,~invariant) linear order. Moreover, let us recall the following three well-known and basic lemmata. For stronger and more refined results, the reader is referred to~\cite{LinnellRhemtullaRolfsen}.

\begin{lem}\label{lemma:basic.order} Suppose that $\preceq$ is a left-invariant (or right-invariant) partial order on a group~$G$. If $(G,{\preceq})$ admits a least or greatest element, then $\vert G \vert = 1$. \end{lem}

\begin{proof} Without loss of generality, let us assume that $(G,{\preceq})$ admits a least element $g \in G$. In particular, $g \preceq e$. Since $\preceq$ is left-invariant (or right-invariant), this entails that $g^{2} \preceq g$, hence $g^{2} = g$ by hypothesis on $g$, and so $g= e$. Thus, $e$ is the least element of $(G,{\preceq})$. Consequently, if $x \in G$, then $e \preceq x^{-1}$ and therefore $x \preceq e$ by left-invariance (or right-invariance) of $\preceq$, which implies that $x = e$. This shows that $G = \{ e \}$, as desired. \end{proof}

\begin{lem}[cf.~\cite{LinnellRhemtullaRolfsen}, Lemma~1.1]\label{lemma:order.dichotomy} Let $\preceq$ be a left-invariant (or right-invariant) linear order on a group $G$. Then $\preceq$ is discrete if and only if $\preceq$ is not dense. \end{lem}

\begin{proof} The implication ($\Longrightarrow$) is obvious. In order to prove the implication ($\Longleftarrow$), suppose that $\preceq$ is not dense. Then there exist $x,y \in G$ such that $x \ne y$ and \begin{displaymath}
	\{ z \in G \mid x \preceq z \preceq y \} \, = \, \{ x,y \} .
\end{displaymath} Now, if $\preceq$ is left-invariant, then we consider $g \defeq x^{-1}y \in G \setminus \{ e \}$ and conclude that \begin{displaymath}
	\{ h \in G \mid e \preceq h \preceq g \} \, = \, x^{-1} \{ z \in G \mid x \preceq z \preceq y \} \, = \, x^{-1} \{ x,y \} \, = \, \{ e,g \} ;
\end{displaymath} and if $\preceq$ is right-invariant, then we consider $g' \defeq yx^{-1} \in G \setminus \{ e \}$ and observe that \begin{displaymath}
	\{ h \in G \mid e \preceq h \preceq g' \} \, = \, \{ z \in G \mid x \preceq z \preceq y \}x^{-1} \, = \, \{ x,y \}x^{-1} \, = \, \{ e,g' \} .
\end{displaymath} In any case, it follows that $\preceq$ is discrete. \end{proof}

In other words, Lemma~\ref{lemma:order.dichotomy} asserts that any left-invariant (or right-invariant) linear order on a group is either discrete or dense. For the next lemma, let us fix one more piece of notation: as usual, the \emph{center} of a group $G$ is defined to be $Z(G) \defeq \{ g \in G \mid \forall h \in G \colon \, gh = hg \}$.

\begin{lem}[cf.~\cite{LinnellRhemtullaRolfsen}, Theorem~2.1]\label{lemma:order.center} If a group $G$ admits an invariant discrete linear order, then $\vert Z(G) \vert > 1$. \end{lem}

\begin{proof} Suppose that $\preceq$ is an invariant discrete linear order on a group $G$. Then there exists some $g \in G \setminus \{ e \}$ with $\{ h \in G \mid e \preceq h \preceq g \} = \{ e,g \}$. We observe that $g \preceq h^{-1}gh$ for all $h \in G$. Indeed, if $h \in G$, then $e \neq h^{-1}gh$, and $e \preceq h^{-1}gh$ by invariance of $\preceq$, wherefore $g \preceq h^{-1}gh$ due to our choice of $g$ and linearity of $\preceq$. Consequently, if $h \in G$, then both $g \preceq h^{-1}gh$ and $g \preceq hgh^{-1}$, i.e., $h^{-1}gh \preceq g$ by invariance of $\preceq$, whence $g = h^{-1}gh$. That is, $g \in Z(G)$. \end{proof}

Any free group (of finite or infinite rank) is an example of an orderable group, as was shown independently by Neumann~\cite[Corollary~3.3]{neumann} and Vinogradov~\cite{vinogradov}. The reader may consult~\cite[Section~2.1.2]{GroupsOrdersDynamics} for a detailed exposition of Vinogradov's argument, and~\cite{bergman} for an alternative proof. We will combine orderability of $F_{2}$ with Corollary~\ref{corollary2} to deduce our next result.

\begin{prop}\label{proposition:aut.is.not.skew.amenable} Let $\leq$ be a dense linear order on a countably infinite set $X$. Then $\mathrm{Aut}(X,{\leq})$ is not skew-amenable with respect to the topology of pointwise convergence associated with the discrete topology on $X$. \end{prop}

\begin{proof} We start by noting that the subset $X' \defeq \{ y \in X \mid \exists x,z \in X \colon \, x < y < z\}$ is countably infinite (as $\vert X\setminus X' \vert \leq 2$) and $\mathrm{Aut}(X,{\leq})$-invariant, ${\leq'} \defeq {\leq} \cap (X' \times X')$ is a dense linear order on $X'$ without least or greatest element, and the well-defined map \begin{displaymath}
	\mathrm{Aut}(X,{\leq}) \, \longrightarrow \, \mathrm{Aut}(X',{\leq}') , \quad g \, \longmapsto \, g\vert_{X'}
\end{displaymath} is an isomorphism of topological groups. Now, according to the independent work of Neumann~\cite[Corollary~3.3]{neumann} and Vinogradov~\cite{vinogradov}, the free group $F_{2}$ (of rank two) admits an invariant linear order~$\preceq$. Since $Z(F_{2}) = \{ e \}$ and $\preceq$ is invariant, Lemma~\ref{lemma:order.center} asserts that $\preceq$ cannot be discrete, therefore must be dense by Lemma~\ref{lemma:order.dichotomy}. Furthermore, by Lemma~\ref{lemma:basic.order}, $(F_{2},{\preceq})$ neither admits a least, nor a greatest element. Due to a well-known back-and-forth argument (see~\cite[Lemma~6.3.3]{PestovBook}, for instance), it follows that $(X',{\leq}') \cong (F_{2},{\preceq})$. By virtue of this isomorphism, $F_{2}$ admits a free action by automorphisms on the ordered set $(X',{\leq}')$, which entails---by a standard argument---the existence of an embedding $\xi \colon F_{2} \to \mathrm{Aut}(X',{\leq}')$ along with a positive unital linear operator $\Phi \colon \ell^{\infty}(F_{2}) \to \ell^{\infty}(X')$ such that \begin{displaymath}
	\forall g \in F_{2} \ \forall f \in \ell^{\infty}(F_{2}) \colon \qquad \Phi (f \circ \lambda_{g}) \, = \, \Phi(f) \circ \xi (g) \, .
\end{displaymath} Consequently, as $F_{2}$ is non-amenable, $\ell^{\infty}(X')$ must not admit an $\mathrm{Aut}(X',{\leq}')$-invariant mean. Hence, by Corollary~\ref{corollary:discrete.skew.amenability}, $\mathrm{Aut}(X',{\leq}') \cong \mathrm{Aut}(X,{\leq})$ is not skew-amenable. \end{proof}

Finally in this note, we will relate the notoriously open question about amenability of the discrete group $F$ to skew-amenability of semi-direct products of topological groups. To this end, it will be convenient to consider another well-known representation of $F$. In this regard, let us note that the topological group $\mathrm{Aut} \!\left( \mathbb{Z}\! \left[ \tfrac{1}{2} \right],{\leq} \right)$ is isomorphic to $\mathrm{Aut}(D,{\leq})$, likewise endowed with the topology of pointwise convergence arising from the discrete topology on the set $D \defeq [0,1] \cap \mathbb{Z}\!\left[ \tfrac{1}{2} \right]$. To give some more detail, let us define $t_{n} \defeq 1 - \tfrac{1}{2^{n+1}}$ for $n \in \mathbb{N}$ and $t_{n} \defeq \tfrac{1}{2^{1-n}}$ for $n \in \mathbb{Z}\setminus \mathbb{N}$. Now, as remarked in~\cite[Remark~2.5]{haagerup} (see also~\cite[2.3]{kaimanovich}), the map $\kappa \colon (0,1) \to \mathbb{R}$ given by \begin{displaymath}
	\kappa (x) \, \defeq \, \tfrac{x-t_{n}}{t_{n+1}-t_{n}} + n \qquad (x \in [t_{n},t_{n+1}], \, n \in \mathbb{Z}) 
\end{displaymath} is an isomorphism between the linearly ordered sets $((0,1),{\leq})$ and $(\mathbb{R},{\leq})$ satisfying \begin{displaymath}
	\kappa \! \left( (0,1) \cap \mathbb{Z}\! \left[ \tfrac{1}{2} \right] \right) \! \, = \, \mathbb{Z}\! \left[ \tfrac{1}{2} \right] \! ,
\end{displaymath} which entails that the map $\phi \colon \mathrm{Aut} ( D,{\leq} ) \to \mathrm{Aut} \!\left( \mathbb{Z}\! \left[ \tfrac{1}{2} \right],{\leq} \right)$ defined by \begin{displaymath}
	\phi(g)(x) \, \defeq \, \kappa \!\left(g\!\left(\kappa^{-1}(x)\right)\right) \qquad \left( g \in \mathrm{Aut} ( D,{\leq} ), \, x \in \mathbb{Z}\!\left[ \tfrac{1}{2} \right] \right)
\end{displaymath} constitutes an isomorphism of topological groups. The image of $F$ under the inverse isomorphism $\phi^{-1}$ consists precisely of those elements of $\mathrm{Aut} ( D,{\leq} )$ which are piecewise affine (with finitely many pieces) and have all their slopes contained in the set $\! \left. \left\{ 2^{k} \, \right\vert k \in \mathbb{Z} \right\}$. As a matter of course, the topology of pointwise convergence on $F$ induced by its action on $D$ agrees with the (original) topology on $F$ inherited from $\mathrm{Aut} \!\left( \mathbb{Z}\! \left[ \tfrac{1}{2} \right],{\leq} \right)$, so that there will be no ambiguity in referring to \emph{the} topology of pointwise convergence on $F$.

The following result is known to experts.

\begin{thm}\label{theorem:kate} The action of $F$ on $D$ is extensively amenable if and only if the (discrete) group $F$ is amenable. \end{thm}

We include a proof of Theorem~\ref{theorem:kate} in the appendix. Combined with Corollary~\ref{corollary:extensive}, the theorem above naturally leads to our final result.

\begin{cor}\label{corollary:kate} Let $H$ be any non-trivial amenable group. Then the following are equivalent. \begin{itemize}
	\item[$(1)$] $H^{(D)} \! \rtimes F$ is skew-amenable with respect to the product topology arising from the discrete topology on $H^{(D)}$ and the topology of pointwise convergence on $F$.
	\item[$(2)$] The (discrete) group $F$ is amenable.
\end{itemize} \end{cor}

\begin{proof} (2)$\Longrightarrow$(1). Evidently, if $F$ is amenable, then the discrete group $H^{(D)} \! \rtimes F$ is amenable, thus the topological group $H^{(D)} \! \rtimes F$ is skew-amenable.

(1)$\Longrightarrow$(2). Suppose that the topological group $H^{(D)} \! \rtimes F$ is skew-amenable. According to Corollary~\ref{corollary:extensive}, this implies that the action of $F$ on $D$ is extensively amenable, whence the discrete group $F$ must be amenable by Theorem~\ref{theorem:kate}. \end{proof}

\begin{remark} A statement analogous to Corollary~\ref{corollary:kate} applies to subgroups of the group \begin{displaymath}
	\mathrm{IET} \, \defeq \, \bigl\{ g \in \mathrm{Sym}(\mathbb{R}/\mathbb{Z}) \, \big\vert \text{ $g$ right-continuous, } \{ gx - x \mid x \in \mathbb{R}/\mathbb{Z} \} \text{ finite} \bigr\} 
\end{displaymath} of \emph{interval exchange transformations} of $\mathbb{R}/\mathbb{Z}$. More precisely, for any subgroup $G \leq \mathrm{IET}$ and any non-trivial amenable group $H$, the following are equivalent: \begin{enumerate}
	\item[$(1)$] $H^{(\mathbb{R}/\mathbb{Z})} \! \rtimes G$ is skew-amenable with respect to the product topology associated with the discrete topology on $H^{(\mathbb{R}/\mathbb{Z})}$ and the topology of pointwise convergence on $G$ arising from the discrete topology on $\mathbb{R}/\mathbb{Z}$.
	\item[$(2)$] The (discrete) group $G$ is amenable.
\end{enumerate} This is a consequence of~\cite[Proposition~5.3]{JMMS} and Corollary~\ref{corollary:extensive}, by the same of line of reasoning as in the proof of Corollary~\ref{corollary:kate}. \end{remark}

\appendix

\section{Proof of Theorem~\ref{theorem:kate}}

For the sake of completeness, we will include a proof of Theorem~\ref{theorem:kate}. To this end, let us agree on some additional notation. For each $g \in F$, we will consider the \emph{right-hand derivative} \begin{displaymath}
	g_{+}'(x) \, \defeq \, \lim_{D \ni t \to 0+} \frac{g(x+t) - g(x)}{t}
\end{displaymath} of $g$ at $x \in D\setminus \{ 1 \}$, as well as the \emph{left-hand derivative} \begin{displaymath}
	g_{-}'(x) \, \defeq \, \lim_{D \ni t \to 0+} \frac{g(x-t) - g(x)}{-t}
\end{displaymath} of $g$ at $x \in D \setminus \{ 0 \}$, respectively. We define a map $\eta \colon F \to \bigoplus_{D} \mathbb{Z}$ by setting \begin{displaymath}
	\eta (g)(x) \, \defeq \, \begin{cases}
		\, \log_{2} g'_{+}(x) - \log_{2} g'_{-}(x) & ( x \in D \setminus \{ 0,1 \} ) , \smallskip \\
		\, \log_{2} g_{+}'(0) & (x=0) , \smallskip \\
		\, - \log_{2} g_{-}'(1) & (x=1)
	\end{cases}
\end{displaymath} for each $g \in F$. Using the chain rule for left and right derivatives of monotone functions on~$D$, it is straightforward to verify the following.

\begin{lem}[cf.~\cite{kaimanovich}, 4.3, (11.18), p.~323]\label{lemma:cocycle} If $g,h \in F$ and $x \in D$, then \begin{displaymath}
		\eta (gh) (x) \, = \, \eta (g)(hx) + \eta (h)(x) .
\end{displaymath} \end{lem}

Now, as it turns out, $F$ can be embedded into $\mathbb{Z}^{(D)} \! \rtimes F$ in such a way that the resulting action of $F$ on $\mathbb{Z}^{(D)}$ is free:

\begin{prop}\label{proposition:kate} The map \begin{displaymath}
	\iota \colon \, F \, \longrightarrow \, \mathbb{Z}^{(D)} \! \rtimes F , \qquad g \, \longmapsto \, \left(\eta\!\left(g^{-1}\right)\!, g\right) 
\end{displaymath} is an embedding and the associated action\footnote{cf.~Remark~\ref{remark:semidirect.product}(3)} \begin{displaymath}
	F \times \mathbb{Z}^{(D)} \! \, \longrightarrow \, \mathbb{Z}^{(D)} , \qquad (g,f) \, \longmapsto \, g\boldsymbol{.}f + \eta\!\left( g^{-1} \right)
\end{displaymath} is free (i.e., has only trivial stabilizers). \end{prop}

\begin{proof} For any two elements $g,h \in F$, \begin{align*}
	\iota (gh) \, &= \, \left(\eta \!\left( (gh)^{-1}\right)\!, gh \right) \,= \, \left(\eta\! \left(h^{-1} g^{-1}\right)\!,gh \right) \\
		& \stackrel{\ref{lemma:cocycle}}{=} \, \left(g\boldsymbol{.}\!\left(\eta\!\left(h^{-1}\right)\right)+\eta\!\left(g^{-1}\right)\!, gh \right) \, = \, \left(\eta\!\left(g^{-1}\right)\!,g\right) \cdot \left(\eta\!\left(h^{-1}\right)\!,h\right) \, = \, \iota (g) \iota (h) \, .
\end{align*} Thus, $\iota \colon F \to \mathbb{Z}^{(D)} \! \rtimes F$ is indeed a homomorphism. We will prove freeness of the resulting action of $F$ on $\mathbb{Z}^{(D)}$ by contradiction. To this end, assume that $f \in \mathbb{Z}^{(D)}$ and $g \in F\setminus \{ e \}$ are such that $g\boldsymbol{.}f + \eta\!\left( g^{-1} \right) = f$, which means that $f(x) = f(g^{-1}x) + \eta(g^{-1})(x)$ for all $x \in D$. Since both $g(0) = 0$ and $g(1) = 1$, our hypothesis readily implies that $\eta(g^{-1})(0) = 0 = \eta(g^{-1})(1)$. Now, as $g$ differs from the identity, the finite set $B \defeq \{ x \in D \mid \eta (g)(x) \ne 0 \}$ must be non-empty, and we have $t \defeq \min (B) \in D\setminus \{ 0,1 \}$. It follows that $g(x) = x$ for all $x \in D$ with $x \leq t$. In particular, $g(t) = t$ and hence \begin{displaymath}
	\eta \!\left(g^{-1}\right)\!(t) \, \stackrel{\ref{lemma:cocycle}}{=} \, -\eta(g)\!\left(g^{-1}t\right) \, = \, -\eta(g)(t) \, \ne \, 0 \, .
\end{displaymath} Consequently, we arrive at \begin{displaymath}
	\left( g\boldsymbol{.}f + \eta\!\left( g^{-1} \right)\right) \! (t) \, = \, f\!\left(g^{-1}t\right) \! + \eta \!\left(g^{-1}\right)\!(t) \, = \, f(t) + \eta \! \left(g^{-1}\right)\!(t) \, \ne \, f(t) \, ,
\end{displaymath} which clearly contradicts our hypothesis. Therefore, the above-mentioned action is free, which in turn entails that $\iota$ is injective, as claimed. \end{proof}

\begin{proof}[Proof of Theorem~\ref{theorem:kate}] ($\Longleftarrow$) Evidently, if the group $F$ is amenable, then so is the semi-direct product $(\mathbb{Z}/2\mathbb{Z})^{(D)} \! \rtimes F$, whence $\ell^{\infty}\!\left( (\mathbb{Z}/2\mathbb{Z})^{(D)} \right)$ admits a $(\mathbb{Z}/2\mathbb{Z})^{(D)} \! \rtimes F$-invariant mean.
	
($\Longrightarrow$) Suppose that the action of $F$ on $D$ is extensively amenable. Hence, by~\cite[Theorem~1.3]{JMMS}, there exists a mean $\mu \colon \ell^{\infty}\!\left( \mathbb{Z}^{(D)} \right) \to \mathbb{R}$ invariant under the action \begin{displaymath}
	\left( \mathbb{Z}^{(D)} \! \rtimes F \right) \! \times \mathbb{Z}^{(D)} \! \, \longrightarrow \, \mathbb{Z}^{(D)} , \quad ((f,g),h) \, \longmapsto \, g\boldsymbol{.}h + f .
\end{displaymath} In particular, $\mu$ must be invariant with respect to the action \begin{displaymath}
	\beta \colon \, F \times \mathbb{Z}^{(D)} \! \, \longrightarrow \, \mathbb{Z}^{(D)} , \qquad (g,f) \, \longmapsto \, g\boldsymbol{.}f + \eta\!\left( g^{-1} \right) .
\end{displaymath} By Proposition~\ref{proposition:kate}, the action $\beta $ is free, from which we infer---by a standard argument---the existence of a positive unital linear operator $\Phi \colon \ell^{\infty}(F) \to \ell^{\infty}\bigl( \mathbb{Z}^{(D)} \bigr)$ satisfying \begin{displaymath}
		\forall g \in F \ \forall f \in \ell^{\infty}(F) \colon \qquad \Phi (f \circ \lambda_{g}) \, = \, \Phi(f) \circ \beta_{g} \, .
	\end{displaymath} In turn, the map $\mu \circ \Phi \colon \ell^{\infty}(F) \to \mathbb{R}$ will constitute a left-invariant mean on $\ell^{\infty}(F)$, witnessing amenability of $F$. This completes the proof. \end{proof}

\section*{Acknowledgments}

This research has received funding of NSF Grant DMS-1932552 and NSF CAREER Award DMS-1352173. The authors would like to express their sincere gratitude towards both Vladimir Pestov and Maxime Gheysens for inspiring discussions and numerous insightful comments on earlier versions of this manuscript. Furthermore, the authors are grateful to Jan Pachl for bringing Remark~\ref{remark:skew.rickert} to their attention, as well as to the anonymous referee for their very careful reading of this manuscript and numerous valuable suggestions for its improvement.

%%%%%%%%%%%%%%%%%%%%%%%%%%%%%%%%%%

\end{document}